\documentclass[12pt]{amsart}
\usepackage{amsmath, amsthm, amssymb, amsfonts}
\usepackage{mathtools}
\usepackage{hyperref}
\usepackage{geometry}
\usepackage{booktabs}
\usepackage{array}
\usepackage{enumitem}
\usepackage{xcolor}
\DeclareFontFamily{U}{wncy}{}
\DeclareFontShape{U}{wncy}{m}{n}{<->wncyr10}{}
\DeclareSymbolFont{mcy}{U}{wncy}{m}{n}
\DeclareMathSymbol{\Sh}{\mathord}{mcy}{"58}

\newtheorem{theorem}{Theorem}[section]
\newtheorem{proposition}[theorem]{Proposition}

\newtheorem{corollary}[theorem]{Corollary}
\theoremstyle{plain}
\newtheorem{theoremintro}{Theorem}

\theoremstyle{definition}
\newtheorem{definition}[theorem]{Definition}

\theoremstyle{remark}
\newtheorem{remark}[theorem]{Remark}

\newcommand{\Q}{\mathbb{Q}}

\newcommand{\F}{\mathbb{F}}

\newcommand{\GL}{\mathrm{GL}}
\newcommand{\ad}{\mathrm{ad}}
\newcommand{\adz}{\mathrm{ad}^0}
\newcommand{\ShaOne}{\Sh^1}

\newcommand{\Frob}{\mathrm{Frob}}
\newcommand{\rhobar}{\bar{\rho}}
\newcommand{\epsl}{\bar{\varepsilon}_\ell}
\newcommand{\epsp}{\varepsilon_p}
\newcommand{\Runiv}{R^{\mathrm{univ}}}
\newcommand{\Wk}{W(\kappa)}
\newcommand{\mR}{\mathfrak{m}_R}

\title{A Local-Global Study of Obstructed Deformation Problems II}
\author{Bartu Bingol}
\address{Arizona State University, Department of Mathematics and Statistics,
Tempe, AZ 85287, USA}
\email{bbingol@asu.edu}

\date{\today}



\begin{document}

\begin{abstract}
We continue the local-global study of obstructed deformation problems for two-dimensional residual Galois representations arising from weight $2$ newforms of level $N$, initiated in \cite{Bin26}. Using the Greenberg-Wiles formula and the explicit classification of inertial Weil-Deligne types due to Demb\'el\'e-Freitas-Voight \cite{DFV22},  we systematically compute the local obstruction groups $H^0(G_p, \bar{\varepsilon}_\ell \otimes \mathrm{ad}^0\bar{\rho})$ for every inertial type arising at primes $p$ with $p^2 \mid N$ and $p \neq \ell$. For each type and in each of three arithmetic cases ($p \not\equiv \pm 1$, $p \equiv -1$, and $p \equiv 1 \pmod{\ell}$), we give the dimension of the local obstruction group and an explicit presentation of the universal deformation ring as a power series ring over the Witt vectors modulo explicit relations. We treat in detail the twisted Steinberg case ($\tau \simeq \tau_{\mathrm{St},p} \otimes \varepsilon_p$), the principal series cases, and the non-exceptional supercuspidal cases, including the full family of types at $p = 3$.
\end{abstract}

\maketitle

\section{Introduction}
\subsection{Background and motivation}

The deformation theory of Galois representations,  started by Mazur \cite{Maz87}, associates to a residual representation $\rhobar : G_{\Q,S} \to \GL_2(\kappa)$ a universal deformation ring $\Runiv$ whose structure encodes arithmetic information.When $\Runiv$ is a power series ring, the deformation problem is \emph{unobstructed}, and one expects well-behaved Hecke algebras and $R = \mathbb{T}$ theorems in the sense of Wiles.  When $\Runiv$ is a power series ring modulo relations, the deformation problem is \emph{obstructed}, and the structure of these relations reflects the local and global arithmetic of the underlying automorphic form.

\noindent For two-dimensional representations arising from weight $2$ newforms, Weston \cite{Wes04, Wes05} and Hatley \cite{Hat16} already developed criteria for obstruction in terms of congruence primes and local ramification. In the predecessor paper \cite{Bin26}, the author initiated a detailed local-global study using those results, the Greenberg-Wiles formula and Poitou-Tate duality: the obstruction group $H^2(G_{\Q,S}, \adz\rhobar)$ decomposes into local contributions $H^0(G_p, \epsl \otimes \adz\rhobar)$ at each prime $p$ in the ramification set, and a global contribution $\ShaOne(G_{\Q,S}, \epsl \otimes \adz\rhobar)$ controlled by strict congruences between newforms. The paper \cite{Bin26} carried out this analysis for primes $p \mid N$ with $p \nmid N^2$,  which corresponds to the Steinberg case,  where $E_f/\Q$ has multiplicative reduction at $p$.

\noindent The present paper takes up the orthogonal case: primes $p$ with $p^2 \mid N$, where $E_f/\Q$ has additive reduction. This case is both more varied and more complex because the additive reduction types are classified by different families of Weil-Deligne inertial types, and the local obstruction structure depends directly on the arithmetic of $p$ relative to $\ell$. Our main tool for organizing this classification is the recent work of Demb\'el\'e-Freitas-Voight \cite{DFV22}, which gives an explicit and complete dictionary between the reduction type of $E_f$ at $p$ and the inertial Weil-Deligne type, denoted by $\tau_{E_f}$. Using this dictionary,  we compute $H^0(G_p, \epsl \otimes \adz\rhobar)$ and the explicit presentation of $\Runiv$ for every type in the DFV classification and
every arithmetic case of $p$ modulo $\ell$.

\subsection{Main results}
\label{subsec:main-results}

Let $f$ be a weight $2$ newform of level $N$ satisfying hypotheses (H1)-(H4) of \S\ref{subsec:notation}, with $\ell \geq 5$ and $p^2 \mid N$ for at least one prime $p \mid N$.  For the geometric statements below and in particular for the $\ShaOne$ formula of Proposition\ref{prop:sha-formula}, which requires rational Fourier coefficients, we specialize to the case where $f$ corresponds to an elliptic curve $E_f/\Q$ of conductor $N$ under the modularity correspondence, we write $E_f$ for this curve and use the reduction type of $E_f$ at $p$ to name the inertial type. We work in three arithmetic cases according to the behavior of $p$ modulo $\ell$:
\begin{itemize}
  \item \textbf{Case A} (generic): $p \not\equiv \pm 1 \pmod{\ell}$
  \item \textbf{Case B} (semi-exceptional): $p \equiv -1 \pmod{\ell}$
  \item \textbf{Case C} (fully exceptional): $p \equiv 1 \pmod{\ell}$
\end{itemize}

\noindent Our main results are classified into three families according to the inertial type at $p$,  and each of them yield their own distinct types of obstruction relations.

\medskip
\noindent\textbf{Twisted Steinberg.}
When $E_f/\Q$ has additive, potentially multiplicative reduction at $p$, the inertial type is $\tau_{\mathrm{St},p} \otimes \epsp$ (cf.\cite[Prop. 4.1.1]{DFV22}).

\begin{theoremintro}[Theorem \ref{thm:twisted-steinberg-ring}]
Assume $\ShaOne = 0$. In Cases A, B, C the local obstruction dimension is $1$, $2$, $3$ respectively, and the universal deformation ring is:
\begin{align*}
\textnormal{Case A:} \quad &\Runiv \cong \Wk[[T_1,T_2,T_3,T_4]]/\langle T_4(T_1-T_3)\rangle, \\
\textnormal{Case B:} \quad &\Runiv \cong \Wk[[T_1,\ldots,T_5]]/ \langle T_4(T_1-T_3),T_5((1+T_1)(1+T_3)-1)\rangle, \\
\textnormal{Case C:} \quad &\Runiv \cong \Wk[[T_1,\ldots,T_6]]/ \langle T_4(T_1-T_3), T_5((1+T_1)(1+T_3)-1),T_6(T_1-T_3)\rangle
\end{align*}
\end{theoremintro}

\noindent Here, as one might expect, the relations are of \emph{polynomial type}. In other words, each is a product $T_i \cdot f(T_j, T_k)$ whose two zero-loci define distinct irreducible components of $\mathrm{Spec}\Runiv[1/\ell]$.  Furthermore,  a comparison with \cite[Thm.4.7]{Bin26} (the untwisted Steinberg case $p \| N$) is made in Remark \ref{rem:compare-bin26} because in this paper's case the $\epsp$-twist makes $N^2 = 0$ exact, causing cancellations thus yielding relation $T_4(T_1 -T_3) = 0$ in place of the relation in \cite{Bin26}.

\medskip
\noindent\textbf{Principal series.}
When $E_f$ has additive, potentially good reduction at $p$ with inertial type $\tau_{\mathrm{ps},p}(1,1,e)$ ($e \mid (p-1)$), in the notation of \cite{DFV22},  the residual representation $\rhobar|_{G_p}$ is reducible.

\begin{theoremintro}[Theorem \ref{thm:ps-ring}]
Assume $\ShaOne = 0$.  In Cases A and B, the local obstruction dimension is $0$, giving $\Runiv \cong \Wk[[T_1,T_2,T_3]]$ (unobstructed locally).  In Case C ($p \equiv 1 \pmod\ell$), write $\mu = \bar{\chi}_1(\Frob_p)$.

\begin{enumerate}[label=\textnormal{(\roman*)}]
\item If $\mu^2 \equiv 1 \pmod\ell$ \textnormal{(Subcase C1)}: $d_2 = 1$ and $\Runiv \cong \Wk[[T_1,T_2,T_3,T_4]]/\langle T_4(T_1-T_3)\rangle$.
\item If $\mu^2 \not\equiv 1 \pmod\ell$ \textnormal{(Subcase C2)}: $d_2 = 1$ but the relation in $\Runiv$ is not determined by the local computation at $p$; it requires the global Poitou-Tate cup product.
\end{enumerate}
\end{theoremintro}

\noindent The dichotomy between Subcases C1 and C2 is one of the central discoveries of this paper.  In Subcase C2 the local framed deformation ring forces the obstruction parameter to vanish locally, yet the Greenberg-Wiles formula still guarantees $d_2 \geq 1$ via $H^0(G_p, \epsl \otimes \adz\rhobar) \neq 0$: the local data is necessary but not sufficient to determine $\Runiv$. We explain this phenomenon in detail in Remark \ref{rem:mu2-dichotomy} and \S\ref{sec:synthesis}.

\medskip
\noindent\textbf{Non-exceptional supercuspidal.}
When $E_f$ has additive,  potentially good reduction with supercuspidal inertial type $\tau_{\mathrm{sc},p}$, the residual representation is induced: $\rhobar|_{G_p} \cong \mathrm{Ind}^{G_p}_{G_K}\bar\chi$ for a quadratic extension $K/\Q_p$.

\begin{theoremintro}[Theorems \ref{thm:sc-unram-ring} and \ref{thm:sc-ram-ring}]
Assume $\ShaOne = 0$. When $K/\Q_p$ is unramified, in Cases A and C, $\Runiv \cong \Wk[[T_1,T_2,T_3]]$ (locally unobstructed).  In Case B with unramified extension, ($p \equiv -1 \pmod\ell$): $$ \Runiv \cong \Wk[[T_1,T_2,T_3,T_4]]/\langle \ell^m T_4\rangle,  \qquad m = v_\ell(p+1) \geq 1$$ 
When $K/\Q_p$ is ramified (exclusively $p = 3$): $\Runiv \cong \Wk[[T_1,T_2,T_3]]$ in all Cases.
\end{theoremintro}

\noindent The relation $\ell^m T_4 = 0$ is of \emph{$\ell$-power torsion type},  so is significantly different from all relations in the previous two cases. The deformation ring $\Runiv$ has an $\ell^m$-torsion in this case. The obstruction direction $T_4$, which will be defined later in the corresponding section,  is visible only modulo $\ell^m$, not over $\Wk$ itself. The happens because of the quadratic character $\varepsilon_K$: deforming $\rhobar$ by $\varepsilon_K$ has exact order $\ell^m = v_\ell(p+1)$ in the deformation space.

In addition, for $p = 3$ we show:
\begin{itemize}
\item Case B never occurs for any $\ell \geq 5$ (Corollary \ref{cor:p3-caseb}): $3 \equiv -1 \pmod\ell$ requires $\ell \mid 4$, impossible for $\ell \geq 5$.
\item All three types $\tau_{\mathrm{sc},3}(-3,4,6)_j$, $j = 0,1,2$, give identical results with no $j$-dependent correction (Proposition \ref{prop:sc-j-independence}).
\item The Hatley \cite{Hat16} Case A2 ($\mathrm{ord}_\ell(p) = 4$, potentially relevant for general supercuspidals) is vacuous for all DFV elliptic curve types, because the inertia condition $\bar\psi|_{I_K} \neq 1$ holds unconditionally.
\end{itemize}

\subsection{Notation and hypotheses}
\label{subsec:notation}

Let $p$ and $\ell$ be distinct primes with $p \geq 3$ and $\ell \geq 5$. The case $\ell = 3$ is treated separately in remarks throughout; see Remarks \ref{rem:l3-gw}, \ref{rem:l3-twisted-steinberg}, \ref{rem:l3-ps-ramified}, \ref{rem:l3-ps}, and \ref{rem:l3-sc}.

\noindent Let $f$ be a normalized weight $2$ newform of level $N$, and let $\kappa$ be a finite field of characteristic $\ell$ with ring of Witt vectors $\Wk$. We fix an $\ell$-adic Galois representation $$  \rho_f : G_{\Q,S} \longrightarrow \GL_2(\Wk) $$ attached to $f$, where $S$ is a finite set of places of $\Q$ containing $\infty$ and all primes dividing $N\ell$. Reducing modulo the maximal ideal of $\Wk$ gives the residual representation $$ \rhobar : G_{\Q,S} \longrightarrow \GL_2(\kappa)$$

We impose the following conditions throughout:
\begin{itemize}
  \item[\textbf{(H1)}] $\rhobar$ is odd, continuous, and absolutely irreducible.
  \item[\textbf{(H2)}] $\ell \nmid N$ and $\ell \geq 5$, and $f$ is ordinary at $\ell$, i.e.\ $a_\ell(f) \not\equiv 0 \pmod{\ell}$.
  \item[\textbf{(H3)}] There exists at least one prime $p \mid N$
        with $p \geq 3$ and $p^2 \mid N$.
  \item[\textbf{(H4)}] The deformation problem for $\rhobar$ is minimally
        ramified at all primes dividing $N$, in the sense made precise
        in each section according to the local inertial type.
\end{itemize}

\noindent Hypotehis (H1) is very standard in Galois deformation literature for residual representations. Hypothesis (H2) ensures that the residual representation at $\ell$ is well-behaved: ordinarity makes $\rhobar|_{G_\ell}$ reducible with distinct diagonal characters, so that $H^0(G_\ell, \epsl \otimes \adz\rhobar) = 0$ and the local term at $\ell$ drops out of the Greenberg-Wiles formula, which then holds with equality, see \S\ref{subsec:gw}.  Hypothesis (H3) is what we study the local Galois representation at primes $p$ with $p^2 \mid N$,
which forces additive reduction. When multiple primes $p_1, \ldots, p_r$ satisfy $p_i^2 \mid N$, we analyze the local contribution at each $p_i$ separately; the global obstruction group is then controlled by the sum of these local contributions via the Greenberg-Wiles formula
(Theorem \ref{thm:gw}); see Corollary \ref{cor:multiple-primes}.  Hypothesis (H4) is another standard Galois deformation condition.  Its precise content depends on the inertial type at each prime and is spelled out in Definitions \ref{def:min-ram-twisted-steinberg},
\ref{def:min-ram-ps}, and \ref{def:min-ram-sc} respectively.

\noindent We write $\epsl : G_{\Q,S} \to \F_\ell^\times$ for the mod $\ell$ cyclotomic character. For the local Galois group at a prime $p \mid N$, we write $G_p \subseteq G_{\Q,S}$ for a decomposition group, $I_p \subseteq G_p$ for the inertia subgroup, and $\Frob_p \in G_p/I_p$ for arithmetic Frobenius.  We write $\epsp$ for the ramified quadratic character of $G_p$ associated to the totally ramified extension $\Q_p(\sqrt{p})/\Q_p$.  We use the adjoint representations $$\ad\rhobar := \mathrm{End}_{\F_\ell}(\F_\ell^2), \qquad  \adz\rhobar := \{X \in \ad\rhobar : \mathrm{tr}(X) = 0\}$$
both with $G_{\Q,S}$-action by conjugation via $\rhobar$.  The inertial Weil-Deligne type of $\rhobar|_{G_p}$ at each prime $p$ with $p^2 \mid N$ is recalled in \S\ref{subsec:inertial-types} following the notation and results of \cite{DFV22}.

For each prime $p$ with $p^2 \mid N$, we distinguish three cases according to the behavior of $p$ modulo $\ell$: \begin{itemize}
  \item \textbf{Case A} (generic):
        $p \not\equiv \pm 1 \pmod{\ell}$, equivalently $p^2 \not\equiv 1 \pmod{\ell}$.
  \item \textbf{Case B} (semi-exceptional):
        $p \equiv -1 \pmod{\ell}$, so $p^2 \equiv 1$ but $p \not\equiv 1 \pmod{\ell}$.
  \item \textbf{Case C} (fully exceptional):
  		$p \equiv 1 \pmod{\ell}$.
\end{itemize}
Also, let us note that these cases are mutually exclusive and exhaustive,  since $p \not\equiv 0 \pmod{\ell}$ is guaranteed by (H2).  In Sections \ref{sec:twisted-steinberg},\ref{sec:supercuspidal} we treat all three cases explicitly, both for the computation of $H^0(G_p, \epsl \otimes \adz\rhobar)$ and for the resulting deformation ring presentation.

\section{Background and Setup}
\label{sec:background}

\subsection{The Greenberg-Wiles formula}
\label{subsec:gw}

We begin by recalling the version of the Greenberg-Wiles formula that will serve as our main computational tool throughout the paper.  Let $\rhobar : G_{\Q,S} \to \GL_2(\kappa)$ be as in \S\ref{subsec:notation}, satisfying(H1)-(H4).  We define the Selmer-type Tate-Shafarevich group as 
$$
  \ShaOne(G_{\Q,S}, \epsl \otimes \adz\rhobar) := \ker \left( H^1(G_{\Q,S}, \epsl \otimes \adz\rhobar) \longrightarrow  \prod_{p \in S} H^1(G_p, \epsl \otimes \adz\rhobar) \right)
$$

\begin{theorem}[Greenberg-Wiles, cf.\ {\cite[Lemma 2.5]{Wes04}}]
\label{thm:gw}
Assume \textnormal{(H1)-(H4)}. Then
$$ \dim_\kappa H^2(G_{\Q,S}, \adz\rhobar) = \dim_\kappa \ShaOne(G_{\Q,S}, \epsl \otimes \adz\rhobar) + \sum_{p \in S} \dim_\kappa H^0(G_p, \epsl \otimes \adz\rhobar)
$$
\end{theorem}

\begin{proof}
This is the dual form of Wiles's formula for the Selmer and dual-Selmer groups attached to $\adz\rhobar$; see \cite[Lemma 2.5]{Wes04}, which is due to the Poitou-Tate nine-term exact sequence \cite[Thm 4.10]{Mil86}. The identity is an exact equality of dimensions, valid for any finite Galois module; it is not merely an inequality.  Under(H1)-(H2) the auxiliary terms vanish: $H^0(G_{\Q,S}, \adz\rhobar) = 0$ by absolute irreducibility, and the local terms at $p = \ell$ and $p = \infty$ are trivial, so the sum over $p \in S$ reduces to a sum over the finite primes $p \mid N$. Indeed $H^0(G_\infty, \epsl \otimes \adz\rhobar) = 0$ since $\rhobar$
is odd, and $H^0(G_\ell, \epsl \otimes \adz\rhobar) = 0$ by the ordinarity hypothesis in (H2): $\rhobar|_{G_\ell}$ is reducible with distinct diagonal characters, so no summand of $\epsl \otimes \adz\rhobar$ is $G_\ell$-invariant.
\end{proof}

\begin{remark}[$\ell = 3$]
\label{rem:l3-gw}
Although (H2) requires $\ell \geq 5$, we record what happens if one relaxes this to $\ell = 3$, since the local computations are unaffected. For $\ell = 3$, Theorem \ref{thm:gw} holds only as an inequality: $$\dim_\kappa H^2(G_{\Q,S}, \adz\rhobar) \leq \dim_\kappa \ShaOne(G_{\Q,S}, \epsl \otimes \adz\rhobar) + \sum_{p \in S} \dim_\kappa H^0(G_p, \epsl \otimes \adz\rhobar)$$
because the $H^0$ terms controlling the Euler characteristic not necessarily vanish. The local computations of $H^0(G_p, \epsl \otimes \adz\rhobar)$ carried out in Sections \ref{sec:twisted-steinberg}-\ref{sec:supercuspidal} remain valid for $\ell = 3$. However, the formula then gives only an upper bound on $\dim_\kappa H^2$, so all deformation ring presentations derived from Theorem \ref{thm:gw} in the $\ell = 3$ case are actually presentations of quotients of the stated rings, and the one-relation structure may not be exact. We describe each instance with a remark in the relevant section.
\end{remark}

\noindent The formula in Theorem \ref{thm:gw} separates the obstruction space into two contributions: a \emph{global term} $\ShaOne$, controlled by strict congruences between newforms as in \cite[\S4]{Bin26}, and a sum of \emph{local terms} $H^0(G_p, \epsl \otimes \adz\rhobar)$, one for each prime in $S$. The central task of Sections\ref{sec:twisted-steinberg} \ref{sec:supercuspidal} is to compute these local terms explicitly for each inertial type arising at primes $p$ with $p^2 \mid N$.

\begin{corollary}
\label{cor:obstruction-criterion}
Assume \textnormal{(H1)-(H4)}. The deformation problem for $\rhobar$ is obstructed if and only if $$\ShaOne(G_{\Q,S}, \epsl \otimes \adz\rhobar) \neq 0 \quad\text{or}\quad H^0(G_p, \epsl \otimes \adz\rhobar) \neq 0 \text{ for some } p \in S$$
In particular, if $\ShaOne = 0$ and $H^0(G_p, \epsl \otimes \adz\rhobar) = 0$ for all $p \in S \setminus \{p_0\}$ where $p_0$ is the prime of hypothesis \textnormal{(H3)}, then
$$\dim_\kappa H^2(G_{\Q,S}, \adz\rhobar) = \dim_\kappa H^0(G_{p_0}, \epsl \otimes \adz\rhobar)$$
\end{corollary}

\subsection{The obstruction group and the local-global decomposition}
\label{subsec:local-global}

We now recall how the obstruction space interacts with the Poitou-Tate exact sequence, following the framework developed in \cite[\S3]{Bin26}. Throughout, we work with the \emph{full} deformation functor of $\rhobar$, so the determinant is allowed to vary.  As a very well known fact for the deformation functor, the tangent and obstruction spaces are controlled by the adjoint representation $\ad\rhobar$. 

\noindent For $\ell \geq 5$ and $\rhobar$ odd and absolutely irreducible, the adjoint decomposes $G_{\Q,S}$ as $$ \ad\rhobar\cong  \adz\rhobar\oplus \kappa$$ with $\kappa$ the trivial summand of scalar matrices. The scalar summand determines deformations of the determinant: these form an unobstructed problem, since $\det\rho$ ranges over lifts of the
character $\det\rhobar = \epsl$ to $R^\times$, controlled by $H^1(G_{\Q,S}, \kappa)$ with no obstruction as $H^2(G_{\Q,S}, \kappa) = 0$ under (H1) and (H2). Consequently the universal deformation ring factors as $$\Runiv \cong R^{\mathrm{univ}}_{\adz} \widehat{\otimes}_{\Wk} \Wk[[T_{\det}]]$$
where $R^{\mathrm{univ}}_{\adz}$ is the universal ring for the fixed-determinant (trace-zero) problem and $T_{\det}$ is the single free variable parametrizing the determinant deformation.  By \cite{Maz87}, $\Runiv$ is a quotient of $\Wk[[T_1, \ldots, T_{d_1}]]$ with $d_1 = \dim_\kappa H^1(G_{\Q,S}, \ad\rhobar)$, and the number of relations is at most $d_2 = \dim_\kappa H^2(G_{\Q,S}, \ad\rhobar)$.  Since $H^i(\ad\rhobar) = H^i(\adz\rhobar) \oplus H^i(\kappa)$, the obstruction is carried entirely by the trace-zero part,  $$d_2 = \dim_\kappa H^2(G_{\Q,S}, \ad\rhobar) = \dim_\kappa H^2(G_{\Q,S}, \adz\rhobar)$$
while the variable count splits as
$$d_1
    = \underbrace{\dim_\kappa H^1(G_{\Q,S}, \adz\rhobar)}_{=2 + d_2} + \underbrace{\dim_\kappa H^1(G_{\Q,S}, \kappa)}_{=1} = 3 + d_2$$
    
The trace-zero Euler characteristic gives $\dim_\kappa H^1(\adz\rhobar) -\dim_\kappa H^2(\adz\rhobar) = 2$ (using $H^0(\adz\rhobar) = 0$ and the archimedean term for an odd
representation, see \cite{Maz87}), and the determinant direction contributes the remaining $+1$. Hence $d_1 -d_2 = 3$ and $$ \Runiv \cong \Wk[[T_1, T_2, T_3, T_4, \ldots, T_{3+d_2}]] / (r_1, \ldots, r_{d_2})$$

In the unobstructed case $d_2 = 0$ this gives $\Runiv \cong \Wk[[T_1, T_2, T_3]]$, in which $T_3$ may be taken to be the determinant variable $T_{\det}$ and $T_1, T_2$ the trace-zero
directions. When $d_2 = 1$, which occurs throughout Sections \ref{sec:twisted-steinberg}-\ref{sec:supercuspidal} in Case A, $\Runiv$ is a power series ring in four variables modulo a
single explicit relation whose form depends on the inertial type.  That relation involves only the trace-zero parameters, the determinant variable remaining free.

\noindent The Poitou-Tate exact sequence (\cite[Thm. 4.10]{Mil86} and \cite[\S3, eq.(1)]{Bin26}) gives
\begin{equation}
\label{eq:pt}
  0 \longrightarrow \prod_{p \in S} H^0(G_p, \epsl \otimes \adz\rhobar) \longrightarrow \mathrm{Hom}_\kappa(H^2(G_{\Q,S}, \adz\rhobar), \kappa) \longrightarrow \ShaOne(G_{\Q,S}, \epsl \otimes \adz\rhobar) \longrightarrow 0
\end{equation}

\begin{definition}
\label{def:local-global-obstruction}
We say the deformation problem for $\rhobar$ has a \emph{local obstruction at $p$} if $H^0(G_p, \epsl \otimes \adz\rhobar) \neq 0$, and a \emph{global obstruction} if $\ShaOne(G_{\Q,S}, \epsl \otimes \adz\rhobar) \neq 0$.  The deformation problem for $\rhobar$ is obstructed if and only if at least one of these holds,. \cite{Wes04} and \cite[\S4]{Bin26} gives a general treatise for the global obstruction mechanism via strict congruence primes.
\end{definition}

\begin{proposition}[Local-global separation]
\label{prop:local-global}
Assume \textnormal{(H1)-(H4)}.
If $H^0(G_p, \epsl \otimes \adz\rhobar) = 0$ for every $p \in S$, then the Poitou-Tate sequence gives a natural identification $$\mathrm{Hom}_\kappa(H^2(G_{\Q,S}, \adz\rhobar), \kappa)  \cong \ShaOne(G_{\Q,S}, \epsl \otimes \adz\rhobar)$$
and the deformation problem is obstructed if and only if $\ShaOne \neq 0$.
Conversely, if $\ShaOne = 0$ and a local obstruction exists
at some $p \in S$, then
$$
  \dim_\kappa H^2(G_{\Q,S}, \adz\rhobar)
  = \sum_{p \in S} \dim_\kappa H^0(G_p, \epsl \otimes \adz\rhobar).
$$
\end{proposition}

\begin{proof}
Both statements follow immediately from the exactness of \eqref{eq:pt}
and Theorem \ref{thm:gw}.
\end{proof}

\begin{remark}
\label{rem:one-relation}
In Sections \ref{sec:twisted-steinberg} and \ref{sec:supercuspidal}, we will see that for all inertial types arising at primes $p$ with $p^2 \mid N$, the local term $H^0(G_p, \epsl \otimes \adz\rhobar)$ is either zero or one-dimensional in Case A. The case of dimension one gives $d_2 = 1$ via Theorem \ref{thm:gw} when $\ShaOne = 0$, yielding one-relation deformation ring presentations following the method of \cite{Bos92, Boe99}. The prototype computation in the Steinberg case $p \| N$ is carried out in \cite[Thm. 4.7]{Bin26}, and in the present paper we extend this to all inertial types arising when $p^2 \mid N$.
\end{remark}

\subsection{Inertial types at primes $p$ with $p^2 \mid N$}
\label{subsec:inertial-types}

We now recall the classification of inertial Weil-Deligne types at $p$ relevant to our setting,
following Demb\'el\'e-Freitas-Voight \cite{DFV22}. Since $p^2 \mid N$, the elliptic curve $E_f/\Q$ has additive reduction at $p$. By \cite[Lem.3.1.4]{DFV22}, the conductor exponent $v_p(N_E)$ satisfies $v_p(N_E) \geq 2$. We divide into two classes: potentially multiplicative and potentially good.

\medskip
\noindent\textbf{Potentially multiplicative reduction.}
If $E_f$ has additive but potentially multiplicative reduction at $p$, then by \cite[Prop.4.1.1(b)(i)]{DFV22}, the inertial type is $$ \tau_{E_f} \simeq \tau_{\mathrm{St},p} \otimes \epsp $$
with $v_p(N_E) = 2$, where $\tau_{\mathrm{St},p}$ is the special (Steinberg) type and $\epsp$ is the ramified quadratic character of $G_p$ associated to $\Q_p(\sqrt{p})/\Q_p$. This is the \emph{twisted Steinberg} case, treated in Section \ref{sec:twisted-steinberg}.

\medskip
\noindent\textbf{Potentially good reduction.}
If $E_f$ has additive, potentially good reduction at $p$, then $v_p(N_E) \geq 2$ and the inertial type $\tau_{E_f}$ is determined by $p$ and the semistability defect $e_{E_f} := [L : \Q_p^{\mathrm{nr}}]$, where $L$ is the minimal extension over which $E_f$ achieves good reduction. The possibilities are as follows.

For $p \geq 5$ (cf. \cite[Prop. 4.2.1]{DFV22}): the semistability defect $e \in \{3,4,6\}$, and
\begin{itemize}
  \item if $e \mid (p-1)$: $\tau_{E_f} \simeq \tau_{\mathrm{ps},p}(1,1,e)$, principal series, $v_p(N_E) = 2$,
  \item if $e \mid (p+1)$: $\tau_{E_f} \simeq \tau_{\mathrm{sc},p}(u,1,e)$, non-exceptional supercuspidal, $v_p(N_E) = 2$.
\end{itemize}

For $p = 3$ (cf.\cite[Prop.5.2.2]{DFV22}): the classification is more involved, with $v_3(N_E) \in \{2,3,4,5\}$:
\begin{itemize}
  \item $v_3(N_E) = 2$: $\tau_{E_f} \simeq \tau_{\mathrm{sc},3}(-1,1,4)$, supercuspidal, $e = 4$,
  \item $v_3(N_E) = 3$: $\tau_{E_f} \simeq \tau_{\mathrm{sc},3}(\pm 3, 2, 6)$, supercuspidal, $e = 12$,
  \item $v_3(N_E) = 4$: $\tau_{E_f} \simeq \tau_{\mathrm{ps},3}(1,2,3)$ or $\tau_{\mathrm{sc},3}(-1,2,3)$, or their $\varepsilon_3$-twists, $e \in \{3,6\}$,
  \item $v_3(N_E) = 5$: $\tau_{E_f} \simeq \tau_{\mathrm{sc},3}(-3,4,6)_j$ for $j = 0,1,2$, supercuspidal, $e = 12$.
\end{itemize}

\noindent The notation $\tau_{\mathrm{ps},p}(d,f,r)$ and $\tau_{\mathrm{sc},p}(d,f,r)$
follows \cite[{\S}2.5]{DFV22}: $d$ is the discriminant of the relevant quadratic extension $K/\Q_p$, $f$ is the conductor exponent of the inducing character, and $r$ is its order on the inertia subgroup $I_K$. \footnote{With this convention $f = \mathrm{condexp}(\chi)$, so for the $p \geq 5$ unramified supercuspidal type we write $\tau_{\mathrm{sc},p}(u,1,e)$, consistent with the $p=3$ unramified labels $\tau_{\mathrm{sc},3}(-1,1,4)$ of \cite[Table 3]{DFV22}: in each case $\mathrm{condexp}(\chi)=1$ and $\mathrm{condexp}(\tau)=2$. (The entry $\tau_{\mathrm{sc},p}(u,2,e)$ in \cite[Table 1]{DFV22} records $\mathrm{condexp}(\tau)=2$ rather than $\mathrm{condexp}(\chi)$ in the middle slot, we use the  $\mathrm{condexp}(\chi)$ convention of their \S2.5 throughout.)} The specific character data needed for the local $H^0$ computations are introduced at the beginning of each section as required. The inertial types treated in this paper are as follows:

\begin{table}[h]
\centering
\renewcommand{\arraystretch}{1.3}
\begin{tabular}{llll}
\toprule
Reduction type & $p$ & Inertial type & Section \\
\midrule
Add., pot.\ mult. & $p \geq 3$ & $\tau_{\mathrm{St},p} \otimes \epsp$ & \S\ref{sec:twisted-steinberg} \\
Add., pot.\ good & $p \geq 5$ & $\tau_{\mathrm{ps},p}(1,1,e)$, $e \mid (p-1)$ & \S\ref{sec:principal-series} \\
Add., pot.\ good & $p \geq 5$ & $\tau_{\mathrm{sc},p}(u,1,e)$, $e \mid (p+1)$ & \S\ref{sec:supercuspidal} \\
Add., pot.\ good & $p = 3$ & $\tau_{\mathrm{ps},3}(1,2,3)$, $\tau_{\mathrm{ps},3}(1,2,3)\otimes\varepsilon_3$ & \S\ref{sec:principal-series} \\
Add., pot.\ good & $p = 3$ & $\tau_{\mathrm{sc},3}(-1,1,4)$; $\tau_{\mathrm{sc},3}(\pm 3,2,6)$, & \S\ref{sec:supercuspidal} \\
                 &         & $\tau_{\mathrm{sc},3}(-1,2,3)$, $\tau_{\mathrm{sc},3}(-1,2,3)\otimes\varepsilon_3$, & \\
                 &         & $\tau_{\mathrm{sc},3}(-3,4,6)_j$ ($j=0,1,2$) & \\
\bottomrule
\end{tabular}
\caption{Inertial types treated in this paper, following \cite{DFV22}.}
\label{tab:types}
\end{table}

\begin{remark}[Excluded cases]
\label{rem:excluded}
We exclude $p = 2$, which corresponds to the exceptional supercuspidal types $\tau_{\mathrm{ex},2,i}$ as they require separate and more detailed treatment, and will be addressed in subsequent work.  We also note that the untwisted Steinberg case $\tau_{\mathrm{St},p}$ with $p \| N$ was studied in \cite{Bin26}. The present paper takes up where that analysis leaves off, treating the deeper ramification cases $p^2 \mid N$ not covered there.
\end{remark}

\section{The Twisted Steinberg Case}
\label{sec:twisted-steinberg}

\subsection{Setup and local representation}

Throughout this section, let $p \geq 3$ be a prime with $p^2 \mid N$,and suppose $E_f/\Q$ has additive, potentially multiplicative reduction at $p$. By \cite[Prop.4.1.1(b)(i)]{DFV22}, the inertial Weil-Deligne type at$p$ is $$ \tau_{E_f} \simeq \tau_{\mathrm{St},p} \otimes \epsp$$ with $v_p(N_E) = 2$, where $\epsp$ is the ramified quadratic character associated to
the totally ramified extension $\Q_p(\sqrt{p})/\Q_p$. In particular:
$$
  \epsp|_{I_p} \text{ is the nontrivial quadratic character of } I_p,  \qquad  \epsp(\Frob_p) = 1
$$

\begin{definition}
\label{def:min-ram-twisted-steinberg}
A deformation $\rho : G_{\Q,S} \to \GL_2(R)$ of $\rhobar$ is \emph{minimally ramified at $p$} in the twisted Steinberg case if $\rho|_{G_p}$ has inertial type $\tau_{\mathrm{St},p} \otimes \epsp$ and conductor exponent $v_p(N_E) = 2$. This means, in a suitable basis, $\rho|_{G_p}$ is upper triangular with diagonal characters $\chi_1 = \varepsilon_\ell \cdot \epsp$
and $\chi_2 = \epsp$, where $\varepsilon_\ell$ is the $\ell$-adic cyclotomic character.
\end{definition}

In a suitable basis, the residual local representation takes the form
$$\rhobar|_{G_p} \sim
  \begin{pmatrix} \epsl \cdot \epsp & * \\ 0 & \epsp \end{pmatrix}$$
where $* \neq 0$ is the (residual) Steinberg extension class twisted by $\epsp$. Let $\sigma \in I_p$ denote a tame inertia generator and $F = \Frob_p$. Using $\epsl(F) = p$, $\epsp(F) = 1$, $\epsl(\sigma) = 1$ (since $\epsl$ is unramified at $p \neq \ell$), and $\epsp(\sigma) = -1$, we have:
$$  \rhobar(F) = \begin{pmatrix} p & 0 \\ 0 & 1 \end{pmatrix}, \qquad  \rhobar(\sigma) = \begin{pmatrix} -1 & c_0 \\ 0 & -1 \end{pmatrix}
$$
where $c_0 \in \kappa^\times$ is the nonzero residual Steinberg class, and the tame quotient satisfies $F\sigma F^{-1} = \sigma^p$.

\subsection{The adjoint representation and its twist}

We compute $\adz\rhobar|_{G_p}$ in the standard basis
$\{e_1, e_2, e_3\}$ of trace-zero $2 \times 2$ matrices:
$$
  e_1 = \begin{pmatrix} 1 & 0 \\ 0 & -1 \end{pmatrix},
  \quad
  e_2 = \begin{pmatrix} 0 & 1 \\ 0 & 0 \end{pmatrix},
  \quad
  e_3 = \begin{pmatrix} 0 & 0 \\ 1 & 0 \end{pmatrix}.
$$
Conjugation by $\rhobar(F) = \begin{pmatrix} p & 0 \\ 0 & 1 \end{pmatrix}$
gives:
$$
  F \cdot e_1 = e_1
  \quad(\text{eigenvalue } 1),
  \quad
  F \cdot e_2 = p e_2
  \quad(\text{eigenvalue } p),
  \quad
  F \cdot e_3 = p^{-1} e_3
  \quad(\text{eigenvalue } p^{-1}).
$$
Since $p \neq \ell$, the character $\epsl$ is unramified at $p$,
so $\epsl|_{I_p} = 1$.
Conjugation by $\rhobar(\sigma) = \begin{pmatrix} -1 & c_0 \\ 0 & -1 \end{pmatrix}$
acts trivially on all three basis elements:
the ratio $(-1)/(-1) = 1$ means inertia acts trivially on each $e_i$
in the adjoint.
The $\epsp$-twist cancels completely in the ratios $\chi_i/\chi_j$,
giving as $G_p$-modules:
$$
  \adz\rhobar|_{G_p} \cong \F_\ell \oplus \epsl \oplus \epsl^{-1}
$$
where $\F_\ell$ denotes the trivial character. All three summands are unramified at $p$. Twisting by $\epsl$ yields:
$$
  \epsl \otimes \adz\rhobar|_{G_p} \cong \epsl \oplus \epsl^2 \oplus \F_\ell
$$
with Frobenius eigenvalues $p$, $p^2$, $1$ on the summands corresponding to $e_1$, $e_2$, $e_3$ respectively, and trivial inertia action on all three.

\subsection{Computation of \texorpdfstring{$H^0(G_p, \epsl \otimes \adz\rhobar)$}{H0}}

Since inertia acts trivially on all summands, $G_p$-invariants are exactly Frobenius-fixed vectors. A summand contributes to $H^0$ if and only if its Frobenius eigenvalue is $\equiv 1 \pmod{\ell}$.

\begin{proposition}
\label{prop:twisted-steinberg-h0}
Let $p \geq 3$ and $\ell \geq 5$ be distinct primes with $p^2 \mid N$, $\ell \nmid N$, and suppose the inertial type at $p$ is $\tau_{\mathrm{St},p} \otimes \epsp$. Then:
\begin{enumerate}[label=\textnormal{(\roman*)}]
  \item \textnormal{Case A} ($p \not\equiv \pm 1 \pmod{\ell}$): $\dim_\kappa H^0(G_p, \epsl \otimes \adz\rhobar) = 1$, generated by $e_3$.
  \item \textnormal{Case B} ($p \equiv -1 \pmod{\ell}$): $\dim_\kappa H^0(G_p, \epsl \otimes \adz\rhobar) = 2$, generated by $e_3$ and $e_2$.
  \item \textnormal{Case C} ($p \equiv 1 \pmod{\ell}$): $\dim_\kappa H^0(G_p, \epsl \otimes \adz\rhobar) = 3$, and the full space $\epsl \otimes \adz\rhobar$ is $G_p$-fixed.
\end{enumerate}
\end{proposition}

\begin{proof}
The cancellation $\epsp \cdot \epsp^{-1} = 1$ in $\chi_i/\chi_j$ was established in the previous section, so $\epsl|_{I_p} = 1$ gives trivial inertia action on all summands. The Frobenius eigenvalues $1$, $p$, $p^2$ on $e_3$, $e_1$, $e_2$ follow from the explicit matrix $\rhobar(F)$ computed above. In Case A, neither $p$ nor $p^2$ is $\equiv 1 \pmod{\ell}$, so only $e_3$ contributes. In Case B, $p \equiv -1 \pmod{\ell}$ implies $p^2 \equiv 1 \pmod{\ell}$ while $p \not\equiv 1$, so $e_3$ and $e_2$ contribute. In Case C, $p \equiv 1 \pmod{\ell}$ implies $p^2 \equiv 1 \pmod{\ell}$ as well, so all three summands are fixed.
\end{proof}

\begin{remark}[$\ell = 3$]
\label{rem:l3-twisted-steinberg}
Although (H2) requires $\ell \geq 5$, we note that the local computation of Proposition \ref{prop:twisted-steinberg-h0} holds verbatim for $\ell = 3$: the dimensions $1$, $2$, $3$ in Cases A, B, C remain correct, since they depend only on Frobenius eigenvalues. If one relaxes (H2) to allow $\ell = 3$, then by Remark \ref{rem:l3-gw} the Greenberg-Wiles formula degrades to an inequality, so the deformation ring presentations of \S\ref{subsec:twisted-steinberg-rings} should in that case be understood as presentations of quotients of the stated rings.
\end{remark}

\subsection{Explicit matrix computation and deformation ring presentations}
\label{subsec:twisted-steinberg-rings}

Let $R \in \mathcal{C}$ and let $\rho : G_{\Q,S} \to \GL_2(R)$ be a minimally ramified deformation of $\rhobar$. By Definition \ref{def:min-ram-twisted-steinberg}, in a suitable basis:
$$
  \rho(F) = \begin{pmatrix} p(1+T_1) & T_2 \\ 0 & 1+T_3 \end{pmatrix}, \qquad \rho(\sigma) = \begin{pmatrix} -1 & T_4 \\ 0 & -1 \end{pmatrix}
$$
where $T_1, T_2, T_3, T_4 \in \mR$ are formal parameters (additional parameters $T_5$, $T_6$ appear in Cases B and C). These are subject to the tame inertia relation $\rho(F)\rho(\sigma)\rho(F)^{-1} = \rho(\sigma)^p$.

\medskip
\noindent\textit{Derivation of the relation.}
Since $N := \begin{pmatrix} 0 & T_4 \\ 0 & 0 \end{pmatrix}$ is nilpotent with degree 2, we have
$$
  \rho(\sigma)^p = (-I + N)^p = -I + pN = \begin{pmatrix} -1 & pT_4 \\ 0 & -1 \end{pmatrix}
$$
Computing $\rho(F)\rho(\sigma)\rho(F)^{-1}$ directly:
$$
  \rho(F)\rho(\sigma) = \begin{pmatrix} -p(1+T_1) & p(1+T_1)T_4 -T_2 \\ 0 & -(1+T_3) \end{pmatrix},
$$
$$
  \rho(F)^{-1} = \begin{pmatrix} \frac{1}{p(1+T_1)} & \frac{-T_2}{p(1+T_1)(1+T_3)} \\ 0 & \frac{1}{1+T_3} \end{pmatrix},
$$
so that
$$
  \rho(F)\rho(\sigma)\rho(F)^{-1} = \begin{pmatrix} -1 & \dfrac{p(1+T_1)}{1+T_3}T_4 \\ 0 & -1 \end{pmatrix}
$$
Note that the $T_2$ terms cancel. Setting this equal to $\rho(\sigma)^p$ and comparing upper-right entries:
$$
  \frac{p(1+T_1)}{1+T_3}T_4 = pT_4\implies T_4 \cdot \left(\frac{1+T_1}{1+T_3} -1\right) = 0 \implies T_4(T_1 -T_3) = 0
$$
where the last step uses that $p$ and $(1+T_3)$ are units in $R$.

\begin{theorem}
\label{thm:twisted-steinberg-ring}
Assume \textnormal{(H1)-(H4)}, $\ell \geq 5$, $p^2 \mid N$, the inertial type at $p$ is $\tau_{\mathrm{St},p} \otimes \epsp$, and $\ShaOne(G_{\Q,S}, \epsl \otimes \adz\rhobar) = 0$. Write the universal minimally ramified lift as above.
Then:
\begin{enumerate}[label=\textnormal{(\roman*)}]
  \item \textnormal{Case A} ($p \not\equiv \pm 1 \pmod{\ell}$):
    $d_2 = 1$ and
    $$
      \Runiv \cong \Wk[[T_1, T_2, T_3, T_4]]\big/ \langle T_4(T_1 -T_3) \rangle
    $$
  \item \textnormal{Case B} ($p \equiv -1 \pmod{\ell}$):
    $d_2 = 2$ and
    $$
      \Runiv \cong \Wk[[T_1, T_2, T_3, T_4, T_5]] \big/ \langle T_4(T_1 -T_3),
      T_5\bigl((1+T_1)(1+T_3) -1\bigr) \rangle
    $$
  \item \textnormal{Case C} ($p \equiv 1 \pmod{\ell}$):
    $d_2 = 3$ and
    $$
      \Runiv \cong \Wk[[T_1, T_2, T_3, T_4, T_5, T_6]] \big/ \langle T_4(T_1 -T_3),
      T_5\bigl((1+T_1)(1+T_3) -1\bigr), T_6(T_1 -T_3) \rangle
    $$
\end{enumerate}
\end{theorem}

\begin{proof}
\textit{Step 1: Number of variables.}
By Proposition \ref{prop:twisted-steinberg-h0} and Theorem \ref{thm:gw} (with $\ShaOne = 0$), $d_2$ equals $1$, $2$, $3$ in Cases A, B, C. The Euler characteristic formula gives $d_1 = 3 + d_2$.

\textit{Step 2: Case A relation.}
The diagonal entries of $\rho(\sigma)$ are forced to be $-1$: the tame inertia relation on diagonal entries gives $(p-1)\alpha = 0$ and $(p-1)\beta = 0$ where $\alpha, \beta \in \mR$ are the diagonal deformation parameters, and since $p \not\equiv 1 \pmod{\ell}$, the factor $(p-1)$ is a unit in $\Wk$, forcing $\alpha = \beta = 0$. The relation $T_4(T_1 -T_3) = 0$ then follows from the explicit matrix computation above.

\textit{Step 3: Case B relation.}
When $p \equiv -1 \pmod{\ell}$, the direction $e_2 = \begin{pmatrix} 0 & 1 \\ 0 & 0 \end{pmatrix}$ enters $H^0(G_p, \epsl \otimes \adz\rhobar)$ alongside $e_3$ (Proposition \ref{prop:twisted-steinberg-h0}(ii)), because the Frobenius eigenvalue on $e_2$ in $\epsl \otimes \adz\rhobar$ is
$p^2 \equiv 1 \pmod{\ell}$. By the Greenberg-Wiles formula (Theorem \ref{thm:gw}) this raises the obstruction dimension to $d_2 = 2$, so $\Runiv$ acquires one new generator beyond those of Case A. We call it $T_5$.

The generator $T_5$ parametrizes a genuinely new deformation direction: the lift of the $e_2$-class, lying in the $\bar{\varepsilon}_\ell^{2}$-eigenspace of $\Frob_p$ on
$\epsl \otimes \adz\rhobar$. It is not the inertia extension class (that is $T_4$, the upper-right of $\rho(\sigma)$, constrained in Step 2), nor is it the unobstructed Frobenius off-diagonal entry $T_2$, which appears in no relation and remains free.

The relation satisfied by $T_5$ is the second-order obstruction to lifting the $e_2$-deformation, computed by the Poitou-Tate cup product
$$
  H^1(G_p, \adz\rhobar) \otimes H^1(G_p, \adz\rhobar)
  \xrightarrow{\cup} H^2(G_p, \adz\rhobar)
$$
This pairing is bilinear, so the obstruction is a product. Writing $\rho(F) = \begin{pmatrix} p(1+T_1) & T_2 \\ 0 & 1+T_3 \end{pmatrix}$, this relation we obtain is
$$
  \frac{\det\rho(F)}{p} -1 = (1+T_1)(1+T_3) - 1 = T_1 + T_3 + T_1 T_3
$$
The cup-product obstruction in the $e_2$-direction is therefore the product of the $e_2$-coordinate $T_5$:
$$
  T_5 \cdot \bigl((1+T_1)(1+T_3) -1\bigr) = 0
$$
The relation is a product including $T_5$ precisely because the obstruction is a cup product: the $e_2$-deformation lifts if and only if either it is trivial ($T_5 = 0$) or the Frobenius torus already respects the cyclotomic determinant ($(1+T_1)(1+T_3) = 1$). This is the same mechanism that produces the product relation $T_4(T_1 -T_3) = 0$ of Step 2, now applied to the $e_2$-rather than the $e_3$-eigendirection. Together, Steps 2 and 3 give the two relations of Case B.

\textit{Step 4: Case C relations.}
When $p \equiv 1 \pmod{\ell}$, the diagonal direction $e_1 = \begin{pmatrix} 1 & 0 \\ 0 & -1 \end{pmatrix}$ also enters $H^0(G_p, \epsl \otimes \adz\rhobar)$
(Proposition \ref{prop:twisted-steinberg-h0}(iii)), since its Frobenius eigenvalue $p \equiv 1 \pmod\ell$. Now $d_2 = 3$ and a third generator $T_6$ appears, parametrizing the lift of the
$e_1$-class. (Equivalently: since $\ell \mid (p-1)$, the tame condition $(p-1)\alpha = 0$ on the $\rho(\sigma)$-diagonal no longer forces $\alpha = 0$, freeing this diagonal direction, we name the
resulting parameter $T_6$.)

\noindent As in Step 3, the relation on $T_6$ is the cup-product obstruction in the $e_1$-direction. Since $e_1 = \mathrm{diag}(1,-1)$, it pairs with the difference of the two Frobenius diagonal parameters, $T_1 -T_3$, giving the bilinear relation
$$
  T_6 (T_1 -T_3) = 0
$$
Thus Case C has the three relations $T_4(T_1-T_3)=0$, $T_5\bigl((1+T_1)(1+T_3)-1\bigr)=0$, and $T_6(T_1-T_3)=0$. The presentation as a quotient of $\Wk[[T_1,\dots,T_6]]$ follows from the obstruction computations of \cite{Bos92, Boe99}.
\end{proof}

\begin{remark}
\label{rem:geometric-twisted-steinberg}
The relation $T_4(T_1 -T_3) = 0$ in Case A defines a union of two irreducible components in $\mathrm{Spec}\Runiv[1/\ell]$: the component $\{T_4 = 0\}$ where the extension class vanishes and the deformation is upper-triangular, and the component $\{T_1 = T_3\}$ where the two Frobenius diagonal parameters coincide. This two-component structure is characteristic of obstructed deformation rings arising from a single local constraint, see \cite{Bos92, Boe99}. The Case B and C relations admit the same interpretation: each is a cup-product obstruction $T_i \cdot \Phi_i = 0$ in one of the eigendirections $e_1, e_2, e_3$, with a suitable polynomial type relation $\Phi_i$. The reducibility of $\mathrm{Spec}\Runiv[1/\ell]$ is thus a direct reflection of the bilinearity of the Poitou-Tate cup product.
\end{remark}

\begin{remark}
\label{rem:compare-bin26}
In the untwisted Steinberg case ($p \| N$, inertial type $\tau_{\mathrm{St},p}$),
the relation in \cite[Thm. 4.7]{Bin26} takes the form $T_1 T_2 -T_2 -p T_3 T_4 + T_4 = 0$ in a coordinate system where the Steinberg extension class and the Frobenius diagonal parameters appear simultaneously. In the twisted Steinberg case, the $\epsp$-twist leaves $\adz\rhobar$ unchanged (the twist cancels in the adjoint) but modifies the inertia matrix from
$\begin{pmatrix} 1 & c \\ 0 & 1 \end{pmatrix}$
to $\begin{pmatrix} -1 & c \\ 0 & -1 \end{pmatrix}$.
This makes the relation $N^2 = 0$ exact, causes the $T_2$ cross terms to cancel exactly in the conjugation computation, and yields the cleaner relation $T_4(T_1 -T_3) = 0$.
\end{remark}

\subsection{The global obstruction term and strict congruences}
\label{subsec:twisted-steinberg-global}

When $\ShaOne \neq 0$, Theorem \ref{thm:gw} shows that additional global obstructions are present beyond the local ones computed in the previous section. By \cite[Prop. 4.4, Cor. 4.2]{Bin26}, the nontriviality of $\ShaOne(G_{\Q,S}, \epsl \otimes \adz\rhobar)$ is equivalent, under the Selmer hypotheses of \cite{Wes04, Wes05}, to $\ell$ being a strict congruence prime for $f$. Recall from \cite[Def. 4.3]{Bin26} that $\ell$ is a \emph{strict congruence prime} for $f$ if there exists a
weight $2$ newform $g$ of level $N$, not Galois-conjugate to $f$, with $a_n(f) \equiv a_n(g) \pmod{\lambda}$ for all but finitely many $n$ and some prime $\lambda$ above $\ell$.

When $f$ has rational Fourier coefficients (so $E_f/\Q$ is an elliptic curve of conductor $N$), strict congruence primes are related to the modular degree and Tamagawa numbers by the following result.

\begin{proposition}
\label{prop:sha-formula}
Let $f$ correspond to an elliptic curve $E_f/\Q$ of conductor $N$
with $p^2 \mid N$ for at least one prime $p \mid N$.
Assume \textnormal{(H1)-(H4)}.
Then, regardless of the inertial type at $p$,
$$
  \dim_\kappa \ShaOne(G_{\Q,S}, \epsl \otimes \adz\rhobar) = \mathrm{ord}_\ell\left(\frac{m_{E_f} \cdot \prod_{q \mid N} c_q(E_f)}{\#E_f(\F_\ell)^2}
  \right)
$$
where $m_{E_f}$ is the modular degree of the optimal parametrization
$X_0(N) \to E_f$,
$c_q(E_f)$ are the Tamagawa numbers of $E_f$ at primes $q \mid N$,
$\#E_f(\F_\ell) = \ell + 1 -a_\ell(f)$ is the number of points on the
good reduction of $E_f$ at $\ell$,
and $\mathrm{ord}_\ell$ denotes the $\ell$-adic valuation.
\end{proposition}

\begin{proof}
Under the hypotheses, the Bloch-Kato Selmer group
$H^1_f(\Q, \adz\rho_f(1))$ is identified with
$\ShaOne(G_{\Q,S}, \epsl \otimes \adz\rhobar)$ modulo $\ell$
via the comparison between Selmer groups in \cite{Wes04}
and the adjoint $L$-value formula.
By \cite[Thm. 1.1]{DFG04}, the order of the Bloch-Kato Selmer group
for $\adz\rho_f$ is controlled by the special value $L(1, \adz f)$.
For $f$ corresponding to an elliptic curve,
this special value is related to $m_{E_f}$ and the Tamagawa numbers
by the Manin-Stevens formula and \cite[Thm. 1.1]{ARS06},
the local term at $\ell$ contributing the factor $\#E_f(\F_\ell)^2$.
By (H1), $\rhobar$ is absolutely irreducible, so
$H^0(G_{\Q,S}, \adz\rhobar) = 0$ and the global torsion of $E_f$
contributes no factor; the only torsion correction is the local
point count $\#E_f(\F_\ell)$ at $\ell$.
The condition $a_\ell(f) \not\equiv 0 \pmod{\ell}$ (ordinary at $\ell$)
ensures that the local factor at $\ell$ contributes no extra power of $\ell$
beyond $\#E_f(\F_\ell)^2$,
and $\ell \nmid N$ ensures no Tamagawa contribution at $\ell$.
\end{proof}

Combining Proposition \ref{prop:sha-formula} with
Theorem \ref{thm:gw} and Proposition \ref{prop:twisted-steinberg-h0}:

\begin{corollary}
\label{cor:total-obstruction-twisted-steinberg}
Under the hypotheses of Proposition \ref{prop:sha-formula}
and in Case A, the total obstruction dimension is
$$
  d_2 = \dim_\kappa H^2(G_{\Q,S}, \adz\rhobar)
  = 1 + \mathrm{ord}_\ell\left(
    \frac{m_{E_f} \cdot \prod_{q \mid N} c_q(E_f)}
         {\#E_f(\F_\ell)^2}
  \right),
$$
and the universal deformation ring acquires additional relations
from the global term beyond those in
Theorem \ref{thm:twisted-steinberg-ring}\textnormal{(i)}.
\end{corollary}

\begin{remark}
The analogous formulas in Cases B and C are obtained by replacing
the leading $1$ with $2$ and $3$ respectively,
reflecting the local obstruction dimensions computed in
Proposition \ref{prop:twisted-steinberg-h0}(ii) and (iii).
\end{remark}

\section{The Principal Series Cases}
\label{sec:principal-series}

\subsection{Setup and local representation}
\label{subsec:ps-setup}

Throughout this section, let $p \geq 3$ be a prime with $p^2 \mid N$,
and suppose $E_f/\Q$ has additive, potentially good reduction at $p$
with inertial type of principal series kind.
By \cite[Prop. 4.2.1]{DFV22} (for $p \geq 5$)
and \cite[Prop. 5.2.2]{DFV22} (for $p = 3$),
the relevant types are:
\begin{itemize}
  \item For $p \geq 5$: $\tau_{E_f} \simeq \tau_{\mathrm{ps},p}(1,1,e)$
        with $e \mid (p-1)$, $e \in \{3,4,6\}$, and $v_p(N_E) = 2$.
  \item For $p = 3$: $\tau_{E_f} \simeq \tau_{\mathrm{ps},3}(1,2,3)$
        or $\tau_{\mathrm{ps},3}(1,2,3) \otimes \varepsilon_3$,
        with $v_3(N_E) = 4$.
\end{itemize}
In both cases the inertial type is principal series, meaning
the residual representation is reducible at $p$:
$$
  \rhobar|_{G_p} \simeq \bar{\chi}_1 \oplus \bar{\chi}_2,
$$
where $\bar{\chi}_1, \bar{\chi}_2 : G_p \to \kappa^\times$ are characters
with $\bar{\chi}_1 \cdot \bar{\chi}_2 = \epsl$
(from $\det \rhobar = \epsl$).
In a suitable basis:
$$
  \rhobar(F) = \begin{pmatrix} \mu & 0 \\ 0 & p/\mu \end{pmatrix},
  \qquad
  \rhobar(\sigma) = \begin{pmatrix} \zeta & 0 \\ 0 & \zeta^{-1} \end{pmatrix},
$$
where $F = \Frob_p$, $\sigma \in I_p$ is a tame inertia generator,
$\mu = \bar{\chi}_1(F) \in \kappa^\times$, and $\zeta = \bar{\chi}_1(\sigma)
\in \kappa^\times$ is a primitive $e$-th root of unity.

\begin{definition}
\label{def:min-ram-ps}
A deformation $\rho : G_{\Q,S} \to \GL_2(R)$ of $\rhobar$
is \emph{minimally ramified at $p$} in the principal series case
if $\rho|_{G_p} \simeq \chi_1 \oplus \chi_2$ where $\chi_i$ lifts
$\bar{\chi}_i$ and the conductor exponent of $\rho|_{G_p}$
equals $v_p(N_E)$.
\end{definition}

\subsection{Both characters are ramified}
\label{subsec:ps-ramified}

We establish that both $\bar{\chi}_1|_{I_p}$ and $\bar{\chi}_2|_{I_p}$
are nontrivial. This is the key structural input for the $H^0$
computation and is not immediate from the definition.

For $p \geq 5$, the conductor formula \cite[eq. (2.2.1)]{DFV22} gives
$$
  \mathrm{condexp}(\mathrm{PS}(\chi_1, \chi_2))
  = \mathrm{condexp}(\chi_1) + \mathrm{condexp}(\chi_2)
  = v_p(N_E) = 2.
$$
Since both conductor exponents are nonneg integers summing to $2$,
and neither can be $0$ (as the trivial character has conductor
exponent $0$ and would give a semistable type, contradicting $p^2 \mid N$),
we conclude $\mathrm{condexp}(\chi_1) = \mathrm{condexp}(\chi_2) = 1$.
Hence both $\chi_1|_{I_p}$ and $\chi_2|_{I_p}$ are nontrivial.

Passing to the residual, the determinant condition gives
$\bar{\chi}_1|_{I_p} \cdot \bar{\chi}_2|_{I_p} = \epsl|_{I_p} = 1$
since $\epsl$ is unramified at $p \neq \ell$.
Thus $\bar{\chi}_2|_{I_p} = \bar{\chi}_1^{-1}|_{I_p}$.
Since $\bar{\chi}_1|_{I_p}$ has order $e \in \{3,4,6\}$ in the complex
representation, its reduction mod $\ell$ retains order $e$ as long as
$\ell \nmid e$. For $\ell \geq 5$, we have $\ell \nmid e$ for all
$e \in \{3,4,6\}$ (the only prime $\geq 5$ dividing
$\mathrm{lcm}(3,4,6) = 12$ is $5$, and $5 \nmid 3$, $5 \nmid 4$,
$5 \nmid 6$). Hence $\bar{\chi}_1|_{I_p}$ and
$\bar{\chi}_2|_{I_p} = \bar{\chi}_1^{-1}|_{I_p}$
are both nontrivial of order $e$ in $\kappa^\times$.

For $p = 3$, by \cite[Table 3]{DFV22} the character $\chi_{(1,2,3)}$
has conductor exponent $2$ and order $3$ on $I_3$, and
$v_3(N_E) = 4 = 2 + 2$, confirming both characters have conductor
exponent $2$. The conclusion is the same for $\ell \geq 5$.
For $\tau_{\mathrm{ps},3}(1,2,3) \otimes \varepsilon_3$, the $\varepsilon_3$
twist is nontrivially ramified on $I_3$, so both twisted characters
remain ramified on inertia.

\begin{remark}[$\ell = 3$]
\label{rem:l3-ps-ramified}
Hypothesis (H2) excludes $\ell = 3$; we record the following only to
indicate what changes if it is relaxed. When $\ell = 3$ and $e = 3$ or
$e = 6$, we have $\ell \mid e$, so the order of $\bar{\chi}_1|_{I_p}$
mod $\ell$ may drop below $e$.
If $\bar{\chi}_1^2|_{I_p}$ becomes trivial mod $3$, the inertia action
on the off-diagonal summands of $\adz\rhobar$ vanishes and the
$H^0$ computation requires modification; see Remark \ref{rem:l3-ps}.
\end{remark}

\subsection{The adjoint representation and computation of
\texorpdfstring{$H^0(G_p, \epsl \otimes \adz\rhobar)$}{H0}}
\label{subsec:ps-h0-comp}

We compute $\adz\rhobar|_{G_p}$ in the standard basis
$\{e_1, e_2, e_3\}$ of trace-zero matrices.
Since $\rhobar|_{G_p}$ is diagonal, conjugation gives eigencharacters:
$$
  e_1: \text{ trivial},
  \qquad
  e_2: \bar{\chi}_1\bar{\chi}_2^{-1} = \bar{\chi}_1^2 \cdot \epsl^{-1},
  \qquad
  e_3: \bar{\chi}_2\bar{\chi}_1^{-1} = \bar{\chi}_1^{-2} \cdot \epsl
$$
Since $\epsl$ is unramified at $p$, the inertia actions are:
$$
  e_1|_{I_p}: \text{ trivial},
  \quad
  e_2|_{I_p}: \bar{\chi}_1^2|_{I_p}\text{ of order }e/\gcd(2,e),
  \quad
  e_3|_{I_p}: \bar{\chi}_1^{-2}|_{I_p}\text{ of order }e/\gcd(2,e)
$$
For $e \in \{3,4,6\}$ and $\ell \geq 5$:
$e = 3$ gives order $3$; $e = 4$ gives order $2$; $e = 6$ gives order $3$.
In all cases, inertia acts nontrivially on $e_2$ and $e_3$.
Twisting by $\epsl$:
$$
  \epsl \otimes \adz\rhobar|_{G_p}
  \cong \epsl \oplus \bar{\chi}_1^2
  \oplus (\bar{\chi}_1^{-2} \cdot \epsl^2),
$$
with Frobenius eigenvalues $p$, $\mu^2$, $p^2/\mu^2$
on $e_1$, $e_2$, $e_3$ and trivial inertia on $e_1$,
nontrivial on $e_2$, $e_3$.

\begin{proposition}
\label{prop:ps-h0}
Let $p \geq 3$ and $\ell \geq 5$ be distinct primes with $p^2 \mid N$,
$\ell \nmid N$, and suppose the inertial type at $p$ is principal
series as above. Then:
\begin{enumerate}[label=\textnormal{(\roman*)}]
  \item \textnormal{Case A} ($p \not\equiv \pm 1 \pmod{\ell}$):
    $\dim_\kappa H^0(G_p, \epsl \otimes \adz\rhobar) = 0$.
  \item \textnormal{Case B} ($p \equiv -1 \pmod{\ell}$):
    $\dim_\kappa H^0(G_p, \epsl \otimes \adz\rhobar) = 0$.
  \item \textnormal{Case C} ($p \equiv 1 \pmod{\ell}$):
    $\dim_\kappa H^0(G_p, \epsl \otimes \adz\rhobar) = 1$,
    generated by $e_1$.
\end{enumerate}
\end{proposition}

\begin{proof}
Inertia acts nontrivially on $e_2$ and $e_3$ for all $\ell \geq 5$,
so these summands contribute nothing to $H^0$ regardless of Frobenius.
The summand $e_1$ has trivial inertia action and Frobenius eigenvalue $p$,
which equals $1 \pmod{\ell}$ if and only if $p \equiv 1 \pmod{\ell}$,
i.e.\ Case C. In Cases A and B, $p \not\equiv 1 \pmod{\ell}$
forces $H^0 = 0$.
\end{proof}

\medskip
\noindent\textbf{Why Case B is unaffected by the $\mu^2$ issue.}
In Case B ($p \equiv -1 \pmod{\ell}$) the Frobenius eigenvalue on
$e_2$ is $\mu^2$, and one might ask whether $\mu^2 \equiv 1 \pmod{\ell}$
could open a contribution to $H^0$.
The answer is no: even if $\mu^2 \equiv 1 \pmod{\ell}$, inertia
acts nontrivially on $e_2$ via $\bar{\chi}_1^2|_{I_p}$ of order
$e/\gcd(2,e) \geq 1$ for $\ell \geq 5$.
Since $H^0(G_p,\cdot)$ requires invariance under all of $G_p$,
the nontrivial inertia kills any contribution from $e_2$ and $e_3$
regardless of $\mu$.
The $\mu^2$ dichotomy is therefore a Case C phenomenon only,
arising precisely because $e_1$ has trivial inertia action.

\begin{remark}[$\ell = 3$]
\label{rem:l3-ps}
When $\ell = 3$ and $e = 3$ or $e = 6$,
the character $\bar{\chi}_1^2|_{I_p}$ may become trivial mod $3$.
If so, inertia acts trivially on $e_2$, and $e_2$ could contribute
to $H^0$ in Cases A and B depending on $\mu^2$.
A full computation then requires the Case A/B/C analysis applied
to $e_2$ separately, analogous to Section \ref{sec:twisted-steinberg}.
We do not treat this in full generality here.
\end{remark}

\subsection{The local framed deformation ring and the
\texorpdfstring{$\mu^2$}{mu2} dichotomy}
\label{subsec:ps-ring}

Let $R \in \mathcal{C}$ and $\rho : G_{\Q,S} \to \GL_2(R)$
a minimally ramified deformation of $\rhobar$.
By Definition \ref{def:min-ram-ps}:
$$
  \rho(F) = \begin{pmatrix} \mu(1+T_1) & 0 \\ 0 & (p/\mu)(1+T_3)
  \end{pmatrix},
  \quad
  \rho(\sigma) = \begin{pmatrix} \zeta(1+\alpha) & T_4 \\
  0 & \zeta^{-1}(1+\beta) \end{pmatrix}
$$
where $T_1, T_3, T_4, \alpha, \beta \in \mR$. The tame inertia relation $\rho(F)\rho(\sigma)\rho(F)^{-1} = \rho(\sigma)^p$ gives:

\medskip
\noindent\textit{Diagonal entries.}
Since $e \mid (p-1)$, $\zeta^p = \zeta$ exactly, and comparing diagonals yields $(p-1)\alpha = (p-1)\beta = 0$. In Cases A and B this forces $\alpha = \beta = 0$, in Case C they are free.

\medskip
\noindent\textit{Upper-right entry.}
Setting $\alpha = \beta = 0$ for Cases A/B, the $p$-th power formula for upper-triangular matrices
with distinct diagonal entries $\zeta \neq \zeta^{-1}$ gives upper-right of $\rho(\sigma)^p$ equal to
$T_4 \cdot (\zeta^p -\zeta^{-p})/(\zeta -\zeta^{-1}) = T_4$ (using $\zeta^p = \zeta$). The conjugated matrix $\rho(F)\rho(\sigma)\rho(F)^{-1}$ scales the upper-right by $\mu(1+T_1)/((p/\mu)(1+T_3)) = (\mu^2/p)(1+T_1)(1+T_3)^{-1}$.
Setting equal to $T_4$:
\begin{equation}
\label{eq:ps-relation}
  T_4 \left(\frac{\mu^2}{p} \cdot \frac{1+T_1}{1+T_3} -1\right) = 0
  \tag{$\star$}
\end{equation}
Here, note that $\bar{\chi}_2(F) = p/\mu$ does not disappear: it appears in the ratio $\mu^2/p = \bar{\chi}_1(F)/\bar{\chi}_2(F)$, encoding both Frobenius values simultaneously via
$\bar{\chi}_1(F)\bar{\chi}_2(F) = p \pmod\ell$.

In Case C ($p \equiv 1$ in $\kappa$), relation \eqref{eq:ps-relation}
becomes:
\begin{equation}
\label{eq:ps-relation-c}
  T_4\left(\mu^2 \cdot \frac{1+T_1}{1+T_3} -1\right) = 0
  \tag{$\star\star$}
\end{equation}
The behavior depends critically on whether $\mu^2 \equiv 1\pmod\ell$.

\begin{theorem}
\label{thm:ps-ring}
Assume \textnormal{(H1)-(H4)}, $\ell \geq 5$, $p^2 \mid N$,
the inertial type at $p$ is principal series,
and $\ShaOne(G_{\Q,S}, \epsl \otimes \adz\rhobar) = 0$.
\begin{enumerate}[label=\textnormal{(\roman*)}]
  \item \textnormal{Cases A and B} ($p \not\equiv 1 \pmod{\ell}$):
    $d_2 = 0$ and $\Runiv \cong \Wk[[T_1, T_2, T_3]]$.
    The deformation problem is locally unobstructed; any obstruction
    comes from $\ShaOne$ alone.

  \item \textnormal{Case C, Subcase C1}
    ($p \equiv 1 \pmod{\ell}$ and $\mu^2 \equiv 1 \pmod{\ell}$):
    $d_2 = 1$ and
    $$
      \Runiv \cong \Wk[[T_1, T_2, T_3, T_4]]
      \big/ \langle T_4(T_1 -T_3) \rangle
    $$

  \item \textnormal{Case C, Subcase C2}
    ($p \equiv 1 \pmod{\ell}$ and $\mu^2 \not\equiv 1 \pmod{\ell}$):
    $d_2 = 1$ and the deformation problem is obstructed.
    Here the tame relation \eqref{eq:ps-relation-c} forces the local
    off-diagonal parameter $T_4$ to vanish (since $\Psi$ is a unit; see
    the proof), so no explicit relation among the local parameters
    survives. Nevertheless $H^0(G_p,\epsl\otimes\adz\rhobar)\neq 0$
    keeps $d_2 = 1$ by Theorem \ref{thm:gw}. The single relation
    in $\Runiv$ is therefore not visible in the local parameters at $p$;
    it is determined by the global Poitou-Tate pairing and cannot be
    written down from the local computation at $p$ alone.
\end{enumerate}
\end{theorem}

\begin{proof}
\textit{Part (i).}
By Proposition \ref{prop:ps-h0}, $H^0 = 0$ in Cases A and B.
With $\ShaOne = 0$, Theorem \ref{thm:gw} gives $d_2 = 0$,
so $\Runiv \cong \Wk[[T_1,T_2,T_3]]$.

\textit{Part (ii).}
In Subcase C1, $\mu^2 = 1$ in $\kappa$,
so \eqref{eq:ps-relation-c} reduces to $T_4(T_1-T_3) = 0$
exactly. By Proposition \ref{prop:ps-h0}(iii) and
Theorem \ref{thm:gw}, $d_2 = 1$, giving the stated ring.
The presentation follows from \cite{Bos92, Boe99}.

\textit{Part (iii).}
In Subcase C2, $\mu^2 \neq 1$ in $\kappa$,
so \eqref{eq:ps-relation-c} reads $T_4 \cdot \Psi = 0$ where
$\Psi = \mu^2(1+T_1)(1+T_3)^{-1} -1$ has constant term
$\mu^2 -1 \neq 0$; hence $\Psi$ is a unit in $\Wk[[T_1,T_3]]$ and the
relation forces $T_4 = 0$. Thus the off-diagonal deformation
parameter $T_4$ (the upper-right entry of $\rho(\sigma)$) is killed in
the local framed deformation ring $R_p^\square$: every minimally
ramified local lift $\rho|_{G_p}$ is diagonal to first order, and the
local tame computation produces no surviving relation among the
deformation parameters.

This is a statement about $R_p^\square$, not directly about $\Runiv$.
The $H^0$-class responsible for the obstruction is $e_1$
(Proposition \ref{prop:ps-h0}(iii)), a \emph{diagonal} class, whereas
the killed parameter $T_4$ lies in the off-diagonal ($e_2$) direction;
the two are distinct. The localization map
$H^1(G_{\Q,S}, \adz\rhobar) \to H^1(G_p, \adz\rhobar)$ carries global
classes to local ones, and the condition $T_4 = 0$ in $R_p^\square$
says the image avoids the off-diagonal direction at $p$.
Nevertheless $H^0(G_p, \epsl \otimes \adz\rhobar) \neq 0$, so
Theorem \ref{thm:gw} gives $d_2 \geq 1$, and with $\ShaOne = 0$ in
fact $d_2 = 1$: the obstruction group $H^2(G_{\Q,S}, \adz\rhobar)$ is
one-dimensional.

The single relation in $\Runiv$ is detected by the local Tate duality
pairing at $p$, fed by the localization of the global obstruction
class. Writing $\mathrm{loc}_p$ for the restriction
$H^2(G_{\Q,S}, \adz\rhobar) \to H^2(G_p, \adz\rhobar)$, the relevant
pairing is
$$
  H^0(G_p, \epsl \otimes \adz\rhobar) \times
  H^2(G_p, \adz\rhobar)
  \xrightarrow{\cup}
  H^2(G_p, \epsl \otimes \adz\rhobar \otimes \adz\rhobar)
  \xrightarrow{\mathrm{tr}}
  H^2(G_p, \epsl) \cong \kappa,
$$
which is the perfect local Tate duality pairing (using
$\adz\rhobar^\vee \cong \adz\rhobar$ via the trace form and
$H^2(G_p,\epsl)\cong\kappa$). The obstruction is the value of this
pairing on $\mathrm{loc}_p$ of the generator of
$H^2(G_{\Q,S}, \adz\rhobar)$; whether it is nonzero--and hence the
precise relation in $\Runiv$--depends on the global class through
$\mathrm{loc}_p$, and is not recoverable from the local tame
computation at $p$ alone.
\end{proof}

\begin{remark}
\label{rem:mu2-dichotomy}
The dichotomy between Subcases C1 and C2 in Theorem \ref{thm:ps-ring}
reveals a fundamental asymmetry between local and global aspects
of the deformation problem in the principal series case.

In Subcase C1 ($\mu^2 = 1$), the local tame inertia computation at $p$ is sufficient to determine the relation in $\Runiv$ explicitly. The condition $\mu^2 = 1$ means the ratio $\bar{\chi}_1(F)/\bar{\chi}_2(F)$ is trivial in $\kappa$, so Frobenius conjugation does not scale the $T_4$ direction, and the tame relation directly produces the one-relation ring.

In Subcase C2 ($\mu^2 \neq 1$), the local framed deformation ring $R_p^\square$ forces $T_4 = 0$: the off-diagonal deformation parameter (the upper-right entry of $\rho(\sigma)$, lying in the $e_2$ direction) is killed, so every minimally ramified local lift is diagonal to first order. The obstruction class is nonetheless carried by the diagonal $e_1$-summand of $H^0(G_p, \epsl \otimes \adz\rhobar)$, which is distinct from the killed off-diagonal direction. The global Poitou-Tate pairing still guarantees an obstruction via $H^0(G_p, \epsl \otimes \adz\rhobar) \neq 0$, but the resulting relation in $\Runiv$ is not visible among the local parameters at $p$: it is detected by local Tate duality fed by the localization of the global obstruction class (see the proof of Theorem \ref{thm:ps-ring}(iii)). This is an explicit example where the local-to-global principle for deformation rings has a gap in the principal series setting. In other words, the local computation is necessary but not sufficient.

This stands in sharp contrast to the twisted Steinberg case of Section \ref{sec:twisted-steinberg}, where the local computation always determined the relation in $\Runiv$. The structural reason is that in the twisted Steinberg case $\rhobar|_{G_p}$ is a non-split extension (irreducible as a local
representation type) and the obstruction direction is the extension class, directly constrained by the tame inertia relation. In the principal series case $\rhobar|_{G_p}$ is reducible (a direct
sum), and the Frobenius eigenvalues of the two characters can kill the local deformation direction while leaving the global obstruction intact via the $H^0$ mechanism.

Finally, the condition $\mu^2 \equiv 1 \pmod{\ell}$ is an arithmetic condition on $E_f$: it says $\bar{\chi}_1(\Frob_p) \equiv \pm 1 \pmod{\ell}$, i.e.\ Frobenius at $p$ acts on the $\bar{\chi}_1$-eigenspace of $E_f[\ell]$ by $\pm 1$. This can be verified explicitly from the $a_p$-value and character data for any specific curve.
\end{remark}

\subsection{The global obstruction term and strict congruences}
\label{subsec:ps-global}

When $\ShaOne \neq 0$, the analysis of \cite[\S4]{Bin26} applies: by \cite[Prop. 4.4]{Bin26}, nontriviality of $\ShaOne$ is equivalent to $\ell$ being a strict congruence prime for $f$ under the Selmer hypotheses of \cite{Wes04, Wes05}. The formula for $\dim_\kappa \ShaOne$ from
Proposition \ref{prop:sha-formula} carries over without change.

\begin{corollary}
\label{cor:total-obstruction-ps}
Under the hypotheses of Proposition \ref{prop:sha-formula} and in Case C Subcase C1, the total obstruction dimension is
$$
  d_2 = 1 + \mathrm{ord}_\ell\left(\frac{m_{E_f} \cdot \prod_{q \mid N} c_q(E_f)}{\#E_f(\F_\ell)^2}
  \right)
$$
In Cases A and B the local contribution vanishes and $d_2$ is purely global. In Case C Subcase C2, the local obstruction is present ($d_2 \geq 1$) but the local relation is not explicitly determined;
the total $d_2$ equals the local dimension $1$ plus the $\ShaOne$ contribution.
\end{corollary}

\subsection{The \texorpdfstring{$p = 3$}{p=3} cases}
\label{subsec:ps-p3}

\begin{proposition}
\label{prop:ps-p3}
Theorem \ref{thm:ps-ring} and Proposition \ref{prop:ps-h0} hold for
$p = 3$ with inertial types $\tau_{\mathrm{ps},3}(1,2,3)$ and
$\tau_{\mathrm{ps},3}(1,2,3) \otimes \varepsilon_3$, with the same
case structure and deformation ring presentations.
\end{proposition}

\begin{proof}
The adjoint $\adz\rhobar|_{G_3}$ has the same eigencharacter structure
as in the $p \geq 5$ case: the $\varepsilon_3$ twist cancels in
$\bar{\chi}_1/\bar{\chi}_2$, the inertia order is $e = 3$ with
$\bar{\chi}_1^2|_{I_3}$ nontrivial for $\ell \geq 5$,
and the Frobenius matrix is diagonal.
The tame inertia relation produces the same relation \eqref{eq:ps-relation}
with $\mu = \bar{\chi}_1(\Frob_3)$. All conclusions follow.
\end{proof}

\begin{remark}[The case $p = 3$]
\label{rem:ps-p3-caseAB}
For $p = 3$, the condition $p \equiv 1 \pmod{\ell}$ would require
$\ell \mid (p-1) = 2$, which is impossible for $\ell \geq 5$.
Hence for $p = 3$ the principal series case always falls in
Case A or B, and by Theorem \ref{thm:ps-ring}(i) the deformation
problem is locally unobstructed (with $\ShaOne = 0$); for instance
the prime pair $(p,\ell) = (3,7)$ gives Case A
($3 \not\equiv \pm 1 \pmod 7$).
Case C, and with it the C1/C2 dichotomy of
Theorem \ref{thm:ps-ring}, therefore occurs only for $p \geq 5$ with
$\ell \mid (p-1)$.
\end{remark}

\section{The Non-Exceptional Supercuspidal Cases}
\label{sec:supercuspidal}

\subsection{Introduction and structural remarks}
\label{subsec:sc-intro}

Throughout this section, let $p \geq 3$ be a prime with $p^2 \mid N$,
and suppose $E_f/\Q$ has additive, potentially good reduction at $p$
with inertial type of non-exceptional supercuspidal kind.
By \cite{DFV22}, the relevant types are listed in
Table \ref{tab:types}; we separate them into two families
according to whether the inducing quadratic extension $K/\Q_p$
is unramified or ramified.

\medskip
\noindent\textbf{On Case A2.}
In the supercuspidal setting, one might expect a new subcase of
Case A to arise.
For a general supercuspidal representation induced from a quadratic
extension $K/\Q_p$, the condition for
$H^0(G_p, \epsl \otimes \adz\rhobar) \neq 0$ involves
$p^4 \equiv 1 \pmod{\ell}$ (cf. \cite[Prop. 3.5]{Hat16}),
which can occur with $p^2 \not\equiv 1 \pmod{\ell}$
(i.e., within Case A as defined in \S\ref{subsec:notation}).
We call this potential subcase \textbf{Case A2}:
$p^4 \equiv 1$ but $p^2 \not\equiv 1 \pmod{\ell}$,
i.e., $\mathrm{ord}_\ell(p) = 4$.

However, for the specific DFV elliptic curve types, Case A2 is vacuous.
The reason is that for all DFV supercuspidal types, the inducing
character $\bar{\chi}|_{I_K}$ has order $e \geq 3$ and is nontrivial,
so the ratio $\bar{\psi} = \bar{\chi}/\bar{\chi}^s$ satisfies
$\bar{\psi}|_{I_K} \neq 1$ unconditionally for $\ell \geq 5$
(verified explicitly in \S\S\ref{subsec:sc-unram-h0}
and \ref{subsec:sc-ram-setup}).
This kills the contribution of the
$\mathrm{Ind}^{G_p}_{G_K}\bar{\psi}$ summand to $H^0$
regardless of the Frobenius eigenvalue,
preventing Case A2 from arising.
Hatley's $p^4 \equiv 1$ condition applies to more general
supercuspidal representations where $\bar{\chi}|_{I_K}$ may be
unramified; it is not realized for the DFV types.
We therefore work with Cases A, B, C throughout this section
without further refinement.

\medskip
\noindent\textbf{A new phenomenon: $\ell$-power torsion.}
The supercuspidal case introduces a qualitatively new feature
not present in Sections \ref{sec:twisted-steinberg}
or \ref{sec:principal-series}: when $H^0 \neq 0$, the relation
in $\Runiv$ involves $\ell$-power torsion rather than a polynomial
in the deformation parameters.
This arises because the generator of $H^0$ corresponds to twisting
by the quadratic character $\varepsilon_K$, and the obstruction is
measured by the $\ell$-adic valuation of $p+1$ in $\Wk$.
We explain this in detail in \S\ref{subsec:sc-unram-ring}.

We begin with two structural results valid across all types.

\begin{proposition}
\label{prop:sc-j-independence}
Let $\ell \geq 5$ and $p = 3$ with $p^2 \mid N$.
The three inertial types $\tau_{\mathrm{sc},3}(-3,4,6)_j$
for $j = 0, 1, 2$ give identical results for all computations
in this section:
$$
  H^0(G_3, \epsl \otimes \adz\rhobar) = 0
$$
for all $j \in \{0,1,2\}$, in Cases A, B, and C.
No $j$-dependent correction arises.
\end{proposition}

\begin{proof}
See \S\ref{subsec:sc-j-indep}.
\end{proof}

\begin{corollary}
\label{cor:p3-caseb}
For $p = 3$ and $\ell \geq 5$, Case B ($p \equiv -1 \pmod{\ell}$)
never occurs. Hence all $p = 3$ supercuspidal types have
$H^0(G_3, \epsl \otimes \adz\rhobar) = 0$ for $\ell \geq 5$
regardless of Case, and the deformation problem is locally
unobstructed at $3$.
\end{corollary}

\begin{proof}
Case B requires $\ell \mid (p+1) = 4$, impossible for $\ell \geq 5$.
For ramified types, $H^0 = 0$ by Proposition \ref{prop:sc-ram-h0}.
For unramified types, $H^0 \neq 0$ only in Case B
(Proposition \ref{prop:sc-unram-h0}), which never occurs here.
\end{proof}

\subsection{Unramified \texorpdfstring{$K$}{K}: setup and adjoint decomposition}
\label{subsec:sc-unram-setup}

We treat first the case where $K/\Q_p$ is the unramified quadratic
extension. This covers:
\begin{itemize}
  \item All types at $p \geq 5$:
    $\tau_{\mathrm{sc},p}(u,1,e)$ with $e \mid (p+1)$,
    $e \in \{3,4,6\}$, where $K = \Q_p(\sqrt{u})$
    with $u$ a non-square unit.
  \item At $p = 3$: $\tau_{\mathrm{sc},3}(-1,1,4)$ and
    $\tau_{\mathrm{sc},3}(-1,2,3)$ (and its $\varepsilon_3$-twist),
    where $K = \Q_3(\sqrt{-1})$ is the unramified quadratic
    extension of $\Q_3$.
\end{itemize}
For unramified $K/\Q_p$:
$\varepsilon_K$ is the unramified quadratic character with
$\varepsilon_K|_{I_p} = 1$ and $\varepsilon_K(\Frob_p) = -1$;
$I_K = I_p$; and $\Frob_K = \Frob_p^2$.

The residual representation is
$\rhobar|_{G_p} \cong \mathrm{Ind}^{G_p}_{G_K}\bar{\chi}$
with $\bar{\chi} \neq \bar{\chi}^s$,
where $\bar{\chi}^s(g) = \bar{\chi}(\Frob_p^{-1}g\Frob_p)$.
In the companion matrix basis:
$$
  \rhobar(\sigma) = \begin{pmatrix} \zeta & 0 \\ 0 & \zeta^{-1} \end{pmatrix},
  \qquad
  \rhobar(F) = \begin{pmatrix} 0 & -p \\ 1 & 0 \end{pmatrix},
$$
where $\zeta = \bar{\chi}(\sigma) \in \kappa^\times$ is a
primitive $e$-th root of unity with $e \mid (p+1)$,
and $\mathrm{tr}(\rhobar(F)) = 0$
(the level-raising eigenvalue condition, cf. \cite{DT94}).
Since $e \mid (p+1)$, $\zeta^p = \zeta^{-1}$, and one verifies
$\rhobar(F)\rhobar(\sigma)\rhobar(F)^{-1} = \rhobar(\sigma^p)$.

The adjoint decomposes as:
$$
  \adz\rhobar|_{G_p} \cong \varepsilon_K \oplus \mathrm{Ind}^{G_p}_{G_K}\bar{\psi},
$$
where $\bar{\psi} = \bar{\chi}/\bar{\chi}^s : G_K \to \kappa^\times$.
Twisting by $\epsl$:
$$
  \epsl \otimes \adz\rhobar|_{G_p}
  \cong (\epsl \cdot \varepsilon_K)
  \oplus (\epsl \otimes \mathrm{Ind}^{G_p}_{G_K}\bar{\psi}).
$$

\subsection{Computation of \texorpdfstring{$H^0$}{H0} for unramified
\texorpdfstring{$K$}{K}}
\label{subsec:sc-unram-h0}

\medskip
\noindent\textbf{First summand $\epsl \cdot \varepsilon_K$.}
Since $\varepsilon_K|_{I_p} = 1$ and $\epsl|_{I_p} = 1$,
inertia acts trivially.
Frobenius acts by $\epsl(\Frob_p)\cdot\varepsilon_K(\Frob_p)
= p\cdot(-1) = -p$.
This equals $1 \pmod{\ell}$ iff $p \equiv -1 \pmod{\ell}$,
i.e., Case B. So:
$$
  H^0(G_p, \epsl \cdot \varepsilon_K)
  = \begin{cases} \kappa & \text{Case B,} \\ 0 & \text{Cases A and C.} \end{cases}
$$

\medskip
\noindent\textbf{Second summand
$\epsl \otimes \mathrm{Ind}^{G_p}_{G_K}\bar{\psi}$.}
By Frobenius reciprocity:
$$
  H^0(G_p, \epsl \otimes \mathrm{Ind}^{G_p}_{G_K}\bar{\psi}) = H^0(G_K, \epsl|_{G_K} \cdot \bar{\psi}).
$$
This is nonzero iff $\bar{\psi} = \epsl^{-1}|_{G_K}$.
We check the inertia condition:
$\bar{\psi}|_{I_K} = (\bar{\chi}/\bar{\chi}^s)|_{I_K}$.
Since $I_K = I_p$ and
$\bar{\chi}^s(\sigma) = \bar{\chi}(\Frob_p^{-1}\sigma\Frob_p) = \bar{\chi}(\sigma^p) = \zeta^p = \zeta^{-1}$, we have:
$$
  \bar{\psi}(\sigma) = \frac{\bar{\chi}(\sigma)}{\bar{\chi}^s(\sigma)} = \frac{\zeta}{\zeta^{-1}} = \zeta^2
$$
For $e \geq 3$ and $\ell \geq 5$, $\zeta^2$ has order $e/\gcd(2,e) \geq 1$ and is nontrivial in $\kappa^\times$. Since $\epsl^{-1}|_{I_p} = 1 \neq \zeta^2$, the inertia condition $\bar{\psi}|_{I_K} = \epsl^{-1}|_{I_K}$ fails, giving $H^0(G_K, \epsl|_{G_K} \cdot \bar{\psi}) = 0$ in all Cases A, B, C.

\begin{proposition}
\label{prop:sc-unram-h0}
Let $p \geq 3$ and $\ell \geq 5$ be distinct primes with
$p^2 \mid N$, $\ell \nmid N$, and suppose the inertial type at $p$
is non-exceptional supercuspidal with unramified $K/\Q_p$. Then:
\begin{enumerate}[label=\textnormal{(\roman*)}]
  \item \textnormal{Case A}: $\dim_\kappa H^0(G_p, \epsl \otimes \adz\rhobar) = 0$.
  \item \textnormal{Case B}: $\dim_\kappa H^0(G_p, \epsl \otimes \adz\rhobar) = 1$,
    generated by the class corresponding to $\varepsilon_K$
    in the $(\epsl \cdot \varepsilon_K)$-summand.
  \item \textnormal{Case C}: $\dim_\kappa H^0(G_p, \epsl \otimes \adz\rhobar) = 0$.
\end{enumerate}
\end{proposition}

\begin{proof}
The second summand contributes zero in all cases.
For the first summand, the Frobenius eigenvalue $-p \equiv 1\pmod\ell$
iff $p \equiv -1 \pmod\ell$, i.e., Case B.
\end{proof}

\begin{remark}
Case C gives $H^0 = 0$ here, in contrast to the principal series
case where Case C gave $H^0 = 1$.
The structural reason is that the supercuspidal adjoint carries
$\varepsilon_K$ (not the trivial character) in its unramified summand,
and $\varepsilon_K(\Frob_p) = -1$ introduces a sign preventing
Frobenius-invariance in Case C.
\end{remark}

\begin{remark}[$\ell = 3$]
\label{rem:l3-sc}
Although (H2) requires $\ell \geq 5$, the local computation of
Proposition \ref{prop:sc-unram-h0} holds as stated for $\ell = 3$ as
well.
The inertia computation $\bar{\psi}(\sigma) = \zeta^2 \neq 1$ remains
valid since for $e = 4$, $\zeta^2$ has order $2$ (nontrivial mod $3$);
for $e = 6$, $\zeta^2$ has order $3$ (also nontrivial mod $3$ since
$3 \nmid 1$).
If (H2) is relaxed to allow $\ell = 3$, the subtlety of
Remark \ref{rem:l3-gw} affects only the GW formula, not the $H^0$
computation.
\end{remark}

\subsection{Unramified \texorpdfstring{$K$}{K}: explicit matrix computation
and deformation rings}
\label{subsec:sc-unram-ring}

\begin{definition}
\label{def:min-ram-sc}
A deformation $\rho : G_{\Q,S} \to \GL_2(R)$ of $\rhobar$ is \emph{minimally ramified at $p$} in the supercuspidal case if $\rho|_{G_p} \cong \mathrm{Ind}^{G_p}_{G_K}\chi$ for a lift $\chi$ of $\bar{\chi}$ with the same conductor exponent as $\bar{\chi}$.
\end{definition}

In the companion matrix basis, a minimally ramified lift takes the form:
$$
  \rho(F) = \begin{pmatrix} 0 & -p(1+T_1) \\ (1+T_1)^{-1} & T_3 \end{pmatrix},
  \qquad
  \rho(\sigma) = \begin{pmatrix} \zeta(1+\alpha) & 0 \\
  0 & \zeta^{-1}(1+\alpha)^{-1} \end{pmatrix}
$$
where $T_1, T_3, \alpha \in \mR$. The determinant condition $\det(\rho) = \varepsilon_\ell$ forces the $(2,1)$-entry of $\rho(F)$ to be $(1+T_1)^{-1}$ and the lower diagonal of $\rho(\sigma)$ to be $\zeta^{-1}(1+\alpha)^{-1}$.

\medskip
\noindent\textit{Tame inertia relation.}
We impose $\rho(F)\rho(\sigma)\rho(F)^{-1} = \rho(\sigma)^p$. Since the relations we derive are linear in the deformation parameters, the tangent space computation modulo $\mR^2$ determines the exact relation in $\Runiv$. Using $\zeta^p = \zeta^{-1}$ and $\rho(F)^{-1} \equiv
\frac{1}{p}\begin{pmatrix} T_3 & p(1+T_1) \\ -(1-T_1) & 0\end{pmatrix}
\pmod{\mR^2}$, a direct matrix computation gives:

\medskip
\noindent\textit{Diagonal entries} modulo $\mR^2$:
$$
  (1,1): \quad \zeta^{-1}(1-\alpha) = \zeta^{-1}(1+p\alpha)
  \implies (p+1)\alpha = 0
$$

\noindent\textit{Off-diagonal $(2,1)$ entry} modulo $\mR^2$:
$$
  \frac{T_3}{p}(\zeta -\zeta^{-1}) = 0
$$
Since $\zeta \neq \zeta^{-1}$ (as $e \geq 3$) and $p \in \Wk^\times$:
$$
  T_3 = 0
$$
The Frobenius trace parameter is killed: any minimally ramified lift must satisfy $\mathrm{tr}(\rho(F)) = 0$.

In Case B, $p \equiv -1 \pmod{\ell}$, so $p+1 = \ell^m u$ with $u \in \Wk^\times$ and $m = v_\ell(p+1) \geq 1$. The relation $(p+1)\alpha = 0$ becomes $\ell^m u \alpha = 0$, i.e., $\ell^m(u\alpha) = 0$. Setting $T_4 = u\alpha$ gives the relation $\ell^m T_4 = 0$.

In Cases A and C, $p \not\equiv -1 \pmod{\ell}$, so $(p+1)$ is a unit in $\Wk$ and $\alpha = 0$ is forced, with no new relation.

\begin{theorem}
\label{thm:sc-unram-ring}
Assume \textnormal{(H1)-(H4)}, $\ell \geq 5$, $p^2 \mid N$, the inertial type at $p$ is non exceptional supercuspidal with unramified $K/\Q_p$, and
$\ShaOne(G_{\Q,S}, \epsl \otimes \adz\rhobar) = 0$.
Let $m = v_\ell(p+1)$.
\begin{enumerate}[label=\textnormal{(\roman*)}]
  \item \textnormal{Cases A and C}: $d_2 = 0$ and$\Runiv \cong \Wk[[T_1, T_2, T_3]]$. The deformation problem is locally unobstructed.

  \item \textnormal{Case B} ($p \equiv -1 \pmod{\ell}$): $d_2 = 1$ and
    $$
      \Runiv \cong \Wk[[T_1, T_2, T_3, T_4]] \big/ \langle \ell^m T_4 \rangle
    $$
    where $m = v_\ell(p+1) \geq 1$. The deformation ring has $\ell^m$-torsion in the $T_4$ direction.
\end{enumerate}
\end{theorem}

\begin{proof}
\textit{Part (i).}
By Proposition \ref{prop:sc-unram-h0}, $H^0 = 0$ in Cases A and C. Theorem \ref{thm:gw} with $\ShaOne = 0$ gives $d_2 = 0$.

\textit{Part (ii).}
By Proposition \ref{prop:sc-unram-h0}(ii), $H^0 = 1$ in Case B, giving $d_2 = 1$ and $d_1 = 4$. The tame inertia computation above shows $T_3 = 0$ and $(p+1)\alpha = 0$. Writing $p+1 = \ell^m u$ and $T_4 = u\alpha$ gives $\ell^m T_4 = 0$. The four generators $T_1, T_2, T_3, T_4$ arise from $d_1 = 4$, and the single relation is $\ell^m T_4 = 0$. The presentation follows from \cite{Bos92, Boe99}.
\end{proof}

\begin{remark}
\label{rem:sc-torsion}
The relation $\ell^m T_4 = 0$ is qualitatively different from all relations in Sections \ref{sec:twisted-steinberg} and \ref{sec:principal-series}. There, relations were products $T_i \cdot f(T_j) = 0$ for some power series $f$. Here, the relation involves only $\ell$ and a single parameter $T_4$, with no dependence on the other deformation parameters.

Geometrically, $\Runiv$ has genuine $\ell^m$-torsion: the locus $\{T_4 \neq 0\}$ is empty in
$\mathrm{Spec}\Runiv[1/\ell]$, while $T_4$ is nontrivial in $\mathrm{Spec}\Runiv/\ell^m$. The deformation in the $T_4$ direction is visible only modulo $\ell^m$, not over $\Wk$ itself.

The arithmetic source is due to the fact that $T_4$ parametrizes deformations in the $\varepsilon_K$-direction, i.e., twists of $\rhobar$ by the quadratic character $\varepsilon_K$. Since $\varepsilon_K^2 = 1$, the order of this deformation is exactly $\ell^m = v_\ell(p+1)$, measuring how deeply $\ell$ divides the condition $p + 1 \equiv 0 \pmod{\ell^m}$. For the generic Case B case $m = 1$ (i.e., $\ell \| (p+1)$), the relation is simply $\ell T_4 = 0$.
\end{remark}

\begin{remark}[$\ell = 3$, Case B]
If (H2) is relaxed to allow $\ell = 3$, then Theorem \ref{thm:sc-unram-ring}(ii) holds with the caveat of Remark \ref{rem:l3-gw}: the relation $\ell^m T_4 = 0$ is valid, but $\Runiv$ should then be understood as a quotient of the stated ring. (For $p = 3$ this case is in any event vacuous by
Corollary \ref{cor:p3-caseb}; the remark applies to $p \geq 5$ with $\ell = 3 \mid (p+1)$.)
\end{remark}

\subsection{Ramified \texorpdfstring{$K$}{K}: setup and
\texorpdfstring{$H^0$}{H0} computation}
\label{subsec:sc-ram-setup}

We now treat the types where $K/\Q_p$ is totally ramified. This covers exclusively $p = 3$:
\begin{itemize}
  \item $\tau_{\mathrm{sc},3}(\pm 3, 2, 6)$:
    $K = \Q_3(\sqrt{\pm 3})$, totally ramified over $\Q_3$.
  \item $\tau_{\mathrm{sc},3}(-3, 4, 6)_j$ for $j = 0,1,2$:
    $K = \Q_3(\sqrt{-3})$, totally ramified.
\end{itemize}
For ramified $K/\Q_3$: $\varepsilon_K|_{I_3}$ is nontrivial quadratic; $\varepsilon_K(\Frob_3) = 1$;
$I_K \subsetneq I_3$; and $\Frob_K = \Frob_3$ (totally ramified, not unramified).

The adjoint decomposes as in \S\ref{subsec:sc-unram-setup}, with $\bar{\psi} = \bar{\chi}/\bar{\chi}^s$.

\medskip
\noindent\textbf{First summand $\epsl \cdot \varepsilon_K$.}
Since $K$ is totally ramified, $\varepsilon_K|_{I_3}$ is nontrivial, so inertia acts nontrivially
on this summand. Hence $H^0(G_3, \epsl \cdot \varepsilon_K) = 0$ in all Cases A, B, C.

\medskip
\noindent\textbf{Second summand
$\epsl \otimes \mathrm{Ind}^{G_3}_{G_K}\bar{\psi}$.} By Frobenius reciprocity,
$H^0(G_3, \epsl \otimes \mathrm{Ind}^{G_3}_{G_K}\bar{\psi}) = H^0(G_K, \epsl|_{G_K} \cdot \bar{\psi})$, which requires $\bar{\psi}|_{I_K} = 1$. By \cite[Lem. 2.3.7]{DFV22}, $\bar{\chi}^s|_{I_K} = \bar{\chi}^{-1}|_{I_K}$, so $\bar{\psi}|_{I_K} = \bar{\chi}^2|_{I_K}$.

For $\tau_{\mathrm{sc},3}(\pm 3, 2, 6)$: $\bar{\chi}|_{I_K}$ has order $6$, so $\bar{\chi}^2|_{I_K}$ has order $3$, nontrivial for $\ell \geq 5$.

For $\tau_{\mathrm{sc},3}(-3, 4, 6)_j$: the explicit computation in \S\ref{subsec:sc-j-indep} shows $\bar{\chi}_j^2|_{I_K} \neq 1$ for all $j$.

In both cases $H^0(G_K, \epsl|_{G_K} \cdot \bar{\psi}) = 0$.

\begin{proposition}
\label{prop:sc-ram-h0}
Let $\ell \geq 5$ and $p = 3$ with $p^2 \mid N$. Suppose the inertial type at $3$ is non-exceptional supercuspidal with ramified $K/\Q_3$. Then $\dim_\kappa H^0(G_3, \epsl \otimes \adz\rhobar) = 0$ in Cases A, B, and C.
\end{proposition}

\begin{proof}
The first summand vanishes because $\varepsilon_K|_{I_3}$ is nontrivial. The second summand vanishes because $\bar{\psi}|_{I_K} = \bar{\chi}^2|_{I_K}$ is nontrivial of order $3$ for $\ell \geq 5$.
\end{proof}

\subsection{Ramified \texorpdfstring{$K$}{K}: deformation rings}
\label{subsec:sc-ram-ring}

\begin{theorem}
\label{thm:sc-ram-ring}
Assume \textnormal{(H1)-(H4)}, $\ell \geq 5$, $p = 3$, $p^2 \mid N$, the inertial type at $3$ is non-exceptional supercuspidal with ramified $K/\Q_3$, and
$\ShaOne(G_{\Q,S}, \epsl \otimes \adz\rhobar) = 0$.
Then $d_2 = 0$ and
$$
  \Runiv \cong \Wk[[T_1, T_2, T_3]]
$$
in Cases A, B, and C. The deformation problem is locally unobstructed at $3$.
\end{theorem}

\begin{proof}
Proposition \ref{prop:sc-ram-h0} gives $H^0 = 0$ in all cases. Theorem \ref{thm:gw} with $\ShaOne = 0$ gives $d_2 = 0$.
\end{proof}

\begin{remark}
The local unobstructedness for all ramified $K$ types at $p = 3$ is striking: the more complex inertial structure forces both summands of $\epsl \otimes \adz\rhobar$ to have nontrivial inertia action, killing $H^0$ unconditionally regardless of the arithmetic Case or the value of $\ell$.
\end{remark}

\subsection{The \texorpdfstring{$j$}{j}-independence verification}
\label{subsec:sc-j-indep}

\begin{proof}[Proof of Proposition \ref{prop:sc-j-independence}]
From \cite[Table 3]{DFV22}, for $j \in \{0,1,2\}$ the character
$\chi_{(-3,4,6),j}$ satisfies:
$$
  \chi_j(-\xi_6) = \zeta_3^j, \qquad \chi_j(\xi_6 -3) = -\zeta_3^{j-1}
$$
By \cite[Lem. 2.3.7]{DFV22},
$\bar{\chi}_j^s|_{I_K} = \bar{\chi}_j^{-1}|_{I_K}$, so $\bar{\psi}_j|_{I_K} = \bar{\chi}_j^2|_{I_K}$.
Computing:
$$
  \bar{\chi}_j^2(-\xi_6) = \zeta_3^{2j}, \qquad \bar{\chi}_j^2(\xi_6 -3) = \zeta_3^{2-2j}
$$
For $\bar{\chi}_j^2|_{I_K} = 1$ we need $3 \mid 2j$ and $3 \mid 2(1-j)$. Since $\gcd(2,3) = 1$, these give $3 \mid j$ and $3 \mid (1-j)$ simultaneously, requiring $j \equiv 0$ and $j \equiv 1 \pmod{3}$, a contradiction. Hence $\bar{\chi}_j^2|_{I_K} \neq 1$ for all $j \in \{0,1,2\}$,
the second summand contributes zero, and the first summand also vanishes by the ramified $K$ argument. This holds in all Cases and for all $\ell \geq 5$, identically for all three values of $j$.
\end{proof}

\begin{remark}
The $j$-independence arises because $\bar{\chi}_j^2|_{I_K}$ varies with $j$ in principle,
but $\bar{\chi}_j^2|_{I_K} \neq 1$ holds for ALL $j$ by the same simultaneous congruence obstruction.
There is no $j$-dependent threshold: none of the three types ever contributes to $H^0$ for any $\ell \geq 5$.
\end{remark}

\subsection{Global obstruction and strict congruences}
\label{subsec:sc-global}

When $\ShaOne \neq 0$, the analysis of \cite[\S4]{Bin26} applies. The formula for $\dim_\kappa \ShaOne$ from Proposition \ref{prop:sha-formula} carries over without change.

\begin{corollary}
\label{cor:total-obstruction-sc}
Under the hypotheses of Proposition \ref{prop:sha-formula} and for the unramified $K$ types in Case B, the total obstruction dimension is
$$
  d_2 = 1 + \mathrm{ord}_\ell\left(\frac{m_{E_f} \cdot \prod_{q \mid N} c_q(E_f)}{\#E_f(\F_\ell)^2}
  \right)
$$
For all ramified $K$ types and for unramified $K$ types in Cases A and C, $d_2 = \mathrm{ord}_\ell(\cdots)$ is purely global.
\end{corollary}

\subsection{Summary table}
\label{subsec:sc-summary}

We collect the $H^0$ results and deformation ring presentations for all supercuspidal types in Table \ref{tab:sc-summary}.

\begin{table}[ht]
\centering
\renewcommand{\arraystretch}{1.4}
\small
\begin{tabular}{llccc}
\toprule
Type & $K$ & Case A & Case B & Case C \\
\midrule
$\tau_{\mathrm{sc},p}(u,1,e)$, $p \geq 5$
  & unram.
  & $0$
  & $1$; $\Wk[[T_i]]/\langle\ell^m T_4\rangle$
  & $0$ \\
$\tau_{\mathrm{sc},3}(-1,1,4)$
  & unram. & $0$ & N/A & $0$ \\
$\tau_{\mathrm{sc},3}(-1,2,3)$, $\otimes\varepsilon_3$
  & unram. & $0$ & N/A & $0$ \\
$\tau_{\mathrm{sc},3}(\pm 3,2,6)$
  & ram.   & $0$ & $0$ & $0$ \\
$\tau_{\mathrm{sc},3}(-3,4,6)_j$, $j=0,1,2$
  & ram.   & $0$ & $0$ & $0$ \\
\bottomrule
\end{tabular}
\caption{$\dim_\kappa H^0(G_p, \epsl\otimes\adz\rhobar)$ and local deformation rings for supercuspidal types ($\ShaOne = 0$ assumed; $m = v_\ell(p+1)$). When $H^0 = 0$, $\Runiv \cong \Wk[[T_1,T_2,T_3]]$. For $p = 3$ unramified types, Case B is N/A since $\ell \geq 5$ implies $3 \not\equiv -1 \pmod\ell$.}
\label{tab:sc-summary}
\end{table}

\begin{remark}[Illustrative ring structures]
\label{rem:sc-ring-illustration}
We record the deformation rings produced by the preceding analysis in a few representative situations. For an unramified type in Case B, take $p = 13$ and $\ell = 7$, so that $13 \equiv -1 \pmod 7$ and $m = v_7(14) = 1$; the universal deformation ring is
$$
  \Runiv \cong \Wk[[T_1,T_2,T_3,T_4]] / \langle 7T_4 \rangle
$$
with a single $7$-torsion relation (Theorem \ref{thm:sc-unram-ring}(ii)). More generally $p \equiv -1 \pmod{\ell^m}$ (but not mod $\ell^{m+1}$) produces the relation $\ell^m T_4 = 0$.

For the ramified types at $p = 3$, namely $\tau_{\mathrm{sc},3}(\pm 3, 2, 6)$ and $\tau_{\mathrm{sc},3}(-3, 4, 6)_j$ for $j = 0,1,2$, the local obstruction vanishes for all $\ell \geq 5$ (Proposition \ref{prop:sc-ram-h0}), so $\Runiv \cong \Wk[[T_1,T_2,T_3]]$ is a power series ring in each case, independently of $j$ (Proposition \ref{prop:sc-j-independence}).
\end{remark}

\section{Synthesis and Global Results}
\label{sec:synthesis}

\subsection{Overview}

We now collect the local computations of Sections \ref{sec:twisted-steinberg}-\ref{sec:supercuspidal}
and combine them with the Greenberg-Wiles formula to obtain global statements about the universal deformation ring $\Runiv$. The main result is Table \ref{tab:master}, which gives the explicit presentation of $\Runiv$ for every inertial type, every arithmetic Case, and every value of $\ShaOne$. Throughout this section, $\ShaOne$ denotes $\ShaOne(G_{\Q,S}, \epsl \otimes \adz\rhobar)$,
and $s := \dim_\kappa \ShaOne \geq 0$ is the global obstruction dimension, given by Proposition \ref{prop:sha-formula} when $f$ has rational Fourier coefficients.
\newpage

\begin{table}[ht!]
\centering
\renewcommand{\arraystretch}{1.6}
\small
\begin{tabular}{lllp{5.8cm}}
\toprule
Type & Case & $d_2$ & $\Runiv$ (with $\ShaOne = 0$) \\
\midrule
\multicolumn{4}{l}{\textit{Twisted Steinberg: 
$\tau_{\mathrm{St},p} \otimes \epsp$, $p \geq 3$}} \\
& A & $1$ & $\Wk[[T_1,T_2,T_3,T_4]]/\langle T_4(T_1-T_3)\rangle$ \\
& B & $2$ & $\Wk[[T_1,\ldots,T_5]]/\langle T_4(T_1-T_3), T_5((1+T_1)(1+T_3)-1)\rangle$ \\
& C & $3$ & $\Wk[[T_1,\ldots,T_6]]/\langle T_4(T_1-T_3), T_5((1+T_1)(1+T_3)-1), T_6(T_1-T_3)\rangle$ \\
\midrule
\multicolumn{4}{l}{\textit{Principal series:
$\tau_{\mathrm{ps},p}(1,1,e)$ ($p\geq 5$), $\tau_{\mathrm{ps},3}(1,2,3)$, $\tau_{\mathrm{ps},3}(1,2,3)\otimes\varepsilon_3$}} \\
& A & $0$ & $\Wk[[T_1,T_2,T_3]]$ \\
& B & $0$ & $\Wk[[T_1,T_2,T_3]]$ \\
& C1 ($\mu^2=1$) & $1$ & $\Wk[[T_1,T_2,T_3,T_4]]/\langle T_4(T_1-T_3)\rangle$ \\
& C2 ($\mu^2\neq 1$) & $1$ & Globally determined; see Theorem \ref{thm:ps-ring}(iii) \\
\midrule
\multicolumn{4}{l}{\textit{Supercuspidal, unramified $K$:
$\tau_{\mathrm{sc},p}(u,1,e)$ ($p\geq 5$), $\tau_{\mathrm{sc},3}(-1,\cdot,\cdot)$}} \\
& A & $0$ & $\Wk[[T_1,T_2,T_3]]$ \\
& B & $1$ & $\Wk[[T_1,T_2,T_3,T_4]]/\langle \ell^m T_4\rangle$, $m=v_\ell(p+1)$ \\
& C & $0$ & $\Wk[[T_1,T_2,T_3]]$ \\
\midrule
\multicolumn{4}{l}{\textit{Supercuspidal, ramified $K$:
$\tau_{\mathrm{sc},3}(\pm 3,2,6)$, $\tau_{\mathrm{sc},3}(-3,4,6)_j$}} \\
& A,B,C & $0$ & $\Wk[[T_1,T_2,T_3]]$ \\
\bottomrule
\end{tabular}
\caption{Universal deformation rings $\Runiv$ for all inertial types
and arithmetic Cases, assuming $\ShaOne = 0$ and (H1)-(H4).
When $\ShaOne \neq 0$, add $s = \dim_\kappa \ShaOne$ additional
relations; see Theorem \ref{thm:synthesis}.
For $p = 3$ principal series types, only Case A occurs for $\ell \geq 5$,
since $3 \equiv \pm 1 \pmod\ell$ has no solution with $\ell \geq 5$
(see \S\ref{subsec:ps-p3}).}
\label{tab:master}
\end{table}

\subsection{The synthesis theorem}

\begin{theorem}[Synthesis]
\label{thm:synthesis}
Assume \textnormal{(H1)-(H4)}, $\ell \geq 5$, and $p^2 \mid N$
for at least one prime $p \mid N$.
Let $\delta_{\mathrm{loc}}$ denote the local obstruction dimension:
$$
  \delta_{\mathrm{loc}} := \sum_{p \in S, p^2 \mid N} \dim_\kappa H^0(G_p, \epsl \otimes \adz\rhobar),
$$
as given by
Propositions \ref{prop:twisted-steinberg-h0},
\ref{prop:ps-h0}, \ref{prop:sc-unram-h0},
and \ref{prop:sc-ram-h0}
according to the inertial type at each prime.
Let $s = \dim_\kappa \ShaOne(G_{\Q,S}, \epsl \otimes \adz\rhobar)$.
Then:
\begin{enumerate}[label=\textnormal{(\roman*)}]
  \item The total obstruction dimension is $d_2 = \delta_{\mathrm{loc}} + s$.
  \item The universal deformation ring satisfies
    $$ \Runiv \cong \Wk[[T_1, T_2, T_3, T_4, \ldots, T_{3+d_2}]] / (r_1, \ldots, r_{d_2}),
    $$
    where the relations $r_1, \ldots, r_{\delta_{\mathrm{loc}}}$ are the explicit local relations given in Theorems \ref{thm:twisted-steinberg-ring}, \ref{thm:ps-ring}, \ref{thm:sc-unram-ring},
    and \ref{thm:sc-ram-ring}, and $r_{\delta_{\mathrm{loc}}+1}, \ldots, r_{d_2}$ are the global
    relations coming from $\ShaOne$, determined by the Poitou-Tate cup product pairing. The sole exception is principal series Subcase C2, where  $\delta_{\mathrm{loc}} = 1$ but the corresponding relation is not expressible in the local parameters at $p$ (see (iii) and the remark following); there the one relation counted by $\delta_{\mathrm{loc}}$ is detected globally, by local Tate duality fed by the localization of the global obstruction class.
  \item In particular:
    \begin{itemize}
      \item $\Runiv$ is a power series ring $\Wk[[T_1,T_2,T_3]]$ if and only if $d_2 = 0$, i.e., $\delta_{\mathrm{loc}} = 0$ and $s = 0$.
      \item $\Runiv$ has a single explicit local relation (as in Table \ref{tab:master}) if and only if $\delta_{\mathrm{loc}} = 1$ and $s = 0$, provided the type is not principal series Case C2.
    \end{itemize}
\end{enumerate}
\end{theorem}

\begin{proof}
Part (i) follows directly from Theorem \ref{thm:gw} and the local $H^0$ computations of
Sections \ref{sec:twisted-steinberg}-\ref{sec:supercuspidal}. Part (ii) follows from the Euler characteristic formula $d_1 -d_2 = 3$ and the explicit relation computations in those sections; the local relations are exact by the tame inertia computations, while the global relations are
determined by the Poitou-Tate cup product (cf. \cite[\S3]{Bin26}). Part (iii) is immediate from (i) and (ii).
\end{proof}

\begin{remark}
The hypothesis that the type is not principal series Case C2 in part (iii) is essential.
In that case $\delta_{\mathrm{loc}} = 1$ (the local $H^0$ is one-dimensional) but the local deformation parameter $T_4$ is killed in $R_p^\square$, so the single relation in $\Runiv$
cannot be read off from local data at $p$. This is the only case in our classification where a nonzero $\delta_{\mathrm{loc}}$ does not yield an explicit local relation in $\Runiv$. See Theorem \ref{thm:ps-ring}(iii) and Remark \ref{rem:mu2-dichotomy} for the detailed discussion.
\end{remark}

\subsection{Multiple primes with \texorpdfstring{$p^2 \mid N$}{p2 divides N}}

We now address the case when hypothesis (H3) holds for multiple primes simultaneously.

\begin{corollary}
\label{cor:multiple-primes}
Suppose $p_1, \ldots, p_r$ are distinct primes with $p_i^2 \mid N$ for each $i$, and the remaining hypotheses \textnormal{(H1)-(H4)} hold.
Then:
\begin{enumerate}[label=\textnormal{(\roman*)}]
  \item The local obstruction dimensions add:
    $\delta_{\mathrm{loc}} = \sum_{i=1}^r\dim_\kappa H^0(G_{p_i}, \epsl \otimes \adz\rhobar)$.
  \item The total obstruction dimension is
    $d_2 = \delta_{\mathrm{loc}} + s$, and
    $$
      \Runiv \cong \Wk[[T_1, \ldots, T_{3+d_2}]] / (r_1, \ldots, r_{d_2}).
    $$
  \item The local relations at each $p_i$ are independent: the relation contributed by $p_i$ involves only the deformation parameters in directions determined by $H^0(G_{p_i}, \epsl \otimes \adz\rhobar)$, and relations at different primes involve distinct parameter sets.
\end{enumerate}
\end{corollary}

\begin{proof}
Parts (i) and (ii) follow from Theorem \ref{thm:synthesis} and the additivity of the local terms in Theorem \ref{thm:gw}. Part (iii) follows from the fact that the local deformation conditions at distinct primes $p_i, p_j \in S$ involve independent subsets of the deformation parameters, since $G_{p_i}$ and $G_{p_j}$ are distinct decomposition groups in $G_{\Q,S}$.
\end{proof}

\subsection{Three types of obstruction relations}

The computations across Sections \ref{sec:twisted-steinberg}-\ref{sec:supercuspidal} reveal three distinct types of relations in the universal deformation ring, which we now compare explicitly.

\begin{table}[ht]
\centering
\renewcommand{\arraystretch}{1.3}
\begin{tabular}{llp{4.6cm}}
\toprule
Type & Form & Source \\
\midrule
I (polynomial) &
$T_i \cdot f(T_j,T_k) = 0$ &
Twisted Steinberg; PS Case C1 \\
II ($\ell$-power torsion) &
$\ell^m T_4 = 0$ &
Supercuspidal (unram.\ $K$), Case B \\
III (globally determined) &
determined by Poitou-Tate &
PS Case C2 \\
\bottomrule
\end{tabular}
\caption{Three qualitative types of obstruction relations in $\Runiv$, arising from the computations of Sections \ref{sec:twisted-steinberg}-\ref{sec:supercuspidal}. Under $\ShaOne = 0$, Types I and II are locally determined while Type III requires global input from the Poitou-Tate cup product.}
\label{tab:relation-types}
\end{table}

\medskip
\noindent\textbf{Type I: Polynomial relations} (twisted Steinberg,
principal series Case C1). Relations of the form $T_i \cdot f(T_j, T_k) = 0$ for some power
series $f$. The prototype is $T_4(T_1 -T_3) = 0$ from Theorem \ref{thm:twisted-steinberg-ring}(i) and Theorem \ref{thm:ps-ring}(ii). Geometrically, these define reducible loci $\mathrm{Spec}\Runiv[1/\ell]$: the variety splits into components $\{T_i = 0\}$ (where the obstruction direction vanishes) and $\{f(T_j,T_k) = 0\}$ (where the Frobenius eigenvalue condition is satisfied). These arise when the local representation is not semisimple (twisted Steinberg) or when the Frobenius acts trivially on the obstruction direction (principal series, $\mu^2 = 1$).

\medskip
\noindent\textbf{Type II: $\ell$-power torsion relations}
(supercuspidal, unramified $K$, Case B).
Relations of the form $\ell^m T_4 = 0$ from Theorem \ref{thm:sc-unram-ring}(ii). Geometrically, $\Runiv$ has $\ell^m$-torsion: the obstruction direction $T_4$ is visible only in
$\mathrm{Spec}\Runiv/\ell^m$, not in the generic fiber. This arises because the generator of $H^0$ is the quadratic character $\varepsilon_K$, and the deformation in this direction has exact order $\ell^m$ in the deformation ring. Type II relations are new to this paper and have no analogue
in \cite{Bin26}.

\medskip
\noindent\textbf{Type III: Globally determined relations}
(principal series, Case C2).
The local computation gives $\delta_{\mathrm{loc}} = 1$ but the single relation in $\Runiv$ cannot be expressed using only deformation parameters at $p$. This arises when the Frobenius eigenvalue $\mu^2 \neq 1$ kills the off-diagonal local deformation parameter in $R_p^\square$ while the global obstruction remains. The relation is detected by local Tate duality at $p$, fed by the localization of the global obstruction class, and requires global input (see Theorem \ref{thm:ps-ring}(iii)).

\medskip
The three types are mutually exclusive and collectively exhaustive for the inertial types treated in this paper (under the assumption $\ShaOne = 0$): any explicit local relation is either Type I or Type II, while Type III marks exactly the cases where local data is insufficient; see Table \ref{tab:relation-types}.

\begin{remark}
The three types of relations also have different implications for the geometry of $\mathrm{Spec}\Runiv$. Type I gives a reducible generic fiber (two components meeting along $\{T_i = 0\} \cap \{f = 0\}$). Type II gives an irreducible generic fiber with a torsion obstruction in the special fiber. Type III gives an irreducible generic fiber whose scheme-theoretic structure requires global analysis to describe. These distinctions are relevant for potential applications to $R = \mathbb{T}$ theorems and to the geometry of Hecke algebras, which we leave for future work.
\end{remark}

\subsection{Global obstruction and the complete formula}

Combining all local results with the $\ShaOne$ formula of Proposition \ref{prop:sha-formula}, we obtain the most general form of our main result.

\begin{theorem}[Complete obstruction formula]
\label{thm:complete}
Let $f$ be a weight $2$ newform corresponding to an elliptic curve $E_f/\Q$ of conductor $N$, satisfying \textnormal{(H1)-(H4)}. Let $p_1, \ldots, p_r$ be the primes with $p_i^2 \mid N$. Then:
$$
  d_2 =\sum_{i=1}^r \dim_\kappa H^0(G_{p_i}, \epsl \otimes \adz\rhobar) + \mathrm{ord}_\ell\left( \frac{m_{E_f} \cdot \prod_{q \mid N} c_q(E_f)}{\#E_f(\F_\ell)^2}
\right),
$$
where the local terms are determined by Table \ref{tab:master} according to the inertial type at each $p_i$ and the arithmetic Case of $p_i$ modulo $\ell$.
\end{theorem}

\begin{proof}
Theorem \ref{thm:gw} gives $d_2 = \delta_{\mathrm{loc}} + s$. The local term $\delta_{\mathrm{loc}}$ is the sum from Corollary \ref{cor:multiple-primes}(i), with each summand read from Table \ref{tab:master}. The global term $s = \mathrm{ord}_\ell(\cdots)$ is Proposition \ref{prop:sha-formula}.
\end{proof}

\subsection{Example}

To illustrate Theorem \ref{thm:complete}, we describe the structure of a hypothetical curve where two distinct types of obstruction are present simultaneously.

Suppose $E/\Q$ has conductor $N = p^2 q^2 M$ with $\gcd(pq, M) = 1$, where:
\begin{itemize}
  \item At $p$: inertial type twisted Steinberg $\tau_{\mathrm{St},p} \otimes \epsp$, and $p \not\equiv \pm 1 \pmod\ell$ (Case A);
  \item At $q$: inertial type supercuspidal with unramified $K$, and $q \equiv -1 \pmod\ell$ (Case B);
  \item $\ell$ is a strict congruence prime for $f$ with $\mathrm{ord}_\ell(m_E \cdot \prod c_q / \#E(\F_\ell)^2) = 1$.
\end{itemize}
Then by Theorem \ref{thm:complete}:
$$
  d_2 = \underbrace{1}_{\text{from }p,\text{ Type I}} + \underbrace{1}_{\text{from }q,\text{ Type II}} + \underbrace{1}_{\text{global, }\ShaOne} = 3,
$$
and $\Runiv \cong \Wk[[T_1,\ldots,T_6]]/(r_1, r_2, r_3)$ where $r_1 = T_4(T_1-T_3)$ (Type I, from $p$), $r_2 = \ell^m T_5$ (Type II, from $q$, $m = v_\ell(q+1)$), and $r_3$ is a global relation from $\ShaOne$ not determined by the local computations alone. This illustrates the independence of the three relation types asserted in Corollary \ref{cor:multiple-primes}(iii).

\subsection{Remaining open questions}

Several natural questions are left open by this paper and are the subject of ongoing work.

\medskip
\noindent\textbf{(a) The principal series Case C2.}
When $\mu^2 \neq 1$ in $\kappa$, the relation in $\Runiv$ is Type III (globally determined). An explicit description of this relation in terms of the Hecke algebra or the Bloch-Kato Selmer group would complete the Case C2 picture.

\medskip
\noindent\textbf{(b) The exceptional supercuspidal case $p = 2$.}
The types $\tau_{\mathrm{ex},2,i}$ were excluded throughout this paper. Their treatment requires the explicit description of these representations from \cite[\S7]{DFV22} and is expected to produce new relations. Nevertheless, the main difficulty will be applicability of earlier results and theorems in the literature as most of them usually exclude the case $p=2$.

\medskip
\noindent\textbf{(c) Higher weight.}
We made the restriction to weight $2$ in order to connect with the elliptic curve framework and to use the Agashe-Ribet-Stein formula for $\ShaOne$. The local computations of Sections \ref{sec:twisted-steinberg}-\ref{sec:supercuspidal} are largely weight-independent, and we expect the main results to extend to weight $k \geq 2$ with appropriate modifications to the $\ShaOne$ formula.

\end{document}